\documentclass[11pt,twoside]{article}
\usepackage[utf8]{inputenc}
\usepackage[T1]{fontenc}
\usepackage{latexsym}
\usepackage{amssymb,amsbsy,amsmath,amsfonts,amssymb,amscd}
\usepackage{cases}
\usepackage{graphicx,color}
\usepackage{float,url}
\usepackage{cancel}
\usepackage[colorlinks=true]{hyperref}
\usepackage{comment}

\newcommand{\commentout}[1]{}
\newcommand{\R}{\mathbb{R}}
\newcommand{\N}{\mathbb{N}}
\newcommand{\Z}{\mathbb{Z}}

\newcommand {\Chi} {{\bf \raise 2pt \hbox{$\chi$}} }
\newcommand {\cae} { {\mathcal E} }

\newcommand {\f}   {\frac}
\newcommand {\p}   {\partial} 

\newcommand {\ov}  {\overline}
\newcommand{\beq}{\begin{equation}}
\newcommand{\eeq}{\end{equation}}
\newcommand{\bea} {\begin{array}{rl}}
\newcommand{\eea} {\end{array}}
\newcommand{\bepa}{\left\{ \begin{array}{l}}
\newcommand{\eepa} {\end{array}\right.}

\newcommand{\norm}[1]{\left\lVert#1\right\rVert}
\newcommand{\abs}[1]{\left\lvert#1\right\rvert}

\newcommand*{\dd}{\mathop{\kern0pt\mathrm{d}}\!{}}

\newcommand{\scal}[2]{\left(#1 , #2\right)}
\newcommand{\lscal}[2]{\left(#1 , #2\right)}

\newcommand\restr[2]{{
  \left.\kern-\nulldelimiterspace 
  #1 
  \vphantom{\big|} 
  \right|_{#2} 
  }}

\newcommand{\kk}{{k+1}}
\newcommand{\kud}{{k+1/2}}

\newlength\eqnspace
\newtheorem{theorem}{Theorem}[section]
\newtheorem{lemma}[theorem]{Lemma}

\newtheorem{remark}[theorem]{Remark}
\newtheorem{proposition}[theorem]{Proposition}

 \newcommand{\AP}[1]{{{#1}}}

\numberwithin{equation}{section}

\newcommand{\qed}{{ \hfill
                       {\unskip\kern 6pt\penalty 500 \raise -2pt\hbox{\vrule\vbox to 6pt{\hrule width 6pt
                       \vfill\hrule}\vrule} \par}   }}
\title{Convergence, error analysis and longtime behavior of the Scalar Auxiliary Variable method for the nonlinear Schrödinger equation}

\date{\today}

\author{ Alexandre Poulain \thanks{Sorbonne Université, CNRS, Université de Paris, Inria, Laboratoire Jacques-Louis Lions (LJLL), F-75005 Paris, France.  email: {alexandre.poulain@sorbonne-universite.fr}} 
 \thanks{The author has received funding from the European Research Council (ERC) under the European Union's Horizon 2020 research and innovation programme (grant agreement No 740623)}
 \and 
 Katharina Schratz \thanks{Sorbonne Université, CNRS, Université de Paris, Laboratoire Jacques-Louis Lions (LJLL), F-75005 Paris, France.  email: {katharina.schratz@sorbonne-universite.fr}} 
   \thanks{The author has received funding from the European Research Council (ERC) under the European Union's Horizon 2020 research and innovation programme (grant agreement No 850941)}
}

\begin{document}
\maketitle
\pagestyle{plain}
\pagenumbering{arabic}

\begin{abstract} 
We carry out the convergence analysis of the Scalar Auxiliary Variable (SAV) method applied to the nonlinear Schr\"odinger equation which preserves a modified Hamiltonian on the discrete level. We derive  a weak and strong convergence result, establish  second-order global error bounds and  present long time error estimates on the modified Hamiltonian. In addition, we illustrate the favorable  energy conservation of the SAV method compared to classical splitting schemes in certain applications.

\end{abstract} 
\vskip .7cm

\noindent{\makebox[1in]\hrulefill}\newline

\section{Introduction}
\begin{sloppypar}
We consider the  Gross-Pitaevskii \cite{gross-struct-1961} equation (NLS) set on the $d$-dimensional torus ${\Omega = \mathbb{T}^d = (\R/2\pi\Z)^d}$ (where $d\le 3$)
\begin{equation}
    i\p_t u(t,x) = -\Delta u(t,x) + V(x)u(t,x) +  f\left(\abs{u(t,x)}^2\right) u(t,x), \quad t\in(0,T]  
    \label{eq:schro}
\end{equation}
with initial conditions $u(0,x)=u^0(x)$, a real-valued interaction potential $V(x)$  and nonlinearity $f(\abs{u}^2)$. 
\end{sloppypar}
The Hamiltonian energy associated to  equation \eqref{eq:schro} takes the form
\begin{equation*}
    H(u,\ov u) = \f12 \int_{\Omega} \left(\abs{\nabla u}^2 + V(x)\abs{u}^2 + F\left(\abs{u}^2\right) \right) \dd x,
\end{equation*}
where $F\left(\abs{u}^2\right)$ is defined by $F'\left(\abs{u}^2\right) =  f\left( \abs{u}^2 \right) $. Note that the Hamiltonian $H(u(t),\ov u(t))$ as well as the probability density $\norm{u(t,\cdot)}^2_{L^2(\Omega)}$ is preserved by the system \eqref{eq:schro}.

In the following we will denote by  $ \cae_1$ the sum of the nonlinear and potential part of the Hamiltonian  
\[
    \cae_1 =  \f12 \int_{\Omega} V(x)\abs{u}^2 +  F\left(\abs{u}^2\right) \dd x.
\]

Using the decomposition $u(t,x) = p(t,x) + i q(t,x) $,   equation \eqref{eq:schro} can be furthermore rewritten as the Hamiltonian system

\begin{equation}
    \begin{cases}
        \p_t p &= -\Delta q + \frac{\delta \cae_1[t]}{\delta q},\\
        \p_t q &= \Delta p - \frac{\delta \cae_1[t]}{\delta p},
    \end{cases}
    \label{eq:SAV-ham}
\end{equation}
with the associated Hamiltonian
\begin{equation}
    H(p,q) = \f12\int_\Omega \abs{\nabla p}^2+ \abs{\nabla q}^2 + V(x)\left(\abs{p}^2+\abs{q}^2\right) + F\left(\abs{p}^2,\abs{q}^2\right) \dd x.
    \label{eq:hamiltonian-energy}
\end{equation}
In this notation, $ \cae_1$  takes the form
\[
    \cae_1 =  \f12 \int_{\Omega} V(x)\left(\abs{p}^2+\abs{q}^2\right) + F\left(\abs{p}^2,\abs{q}^2\right) \dd x.
\]

Due to their importance in numerous applications, reaching from Bose-Einstein condensation over nonlinear optics up to plasma physics, nonlinear Schr\"odinger equations  are nowadays very well studied numerically. In the last decades a large variety of different numerical schemes has been proposed~\cite{Bao-Cai,Antoine-Computational, Bao-numerical,Gonzales,Gonzales-coupled}.
Thanks to their simplicity and accuracy, a popular choice thereby lies in so-called splitting methods, where the right hand side of \eqref{eq:schro} is split into the linear and nonlinear part, respectively, see, e.g.,  \cite{Bao-time-2002,Besse-order-2002,Bao-nonlinear-2007} and the references therein. The  popularity of splitting methods also stems from their structure preservation. They conserve exactly the $L^2$ norm of the solution and  allow for near energy conservation over long times, see, e.g., \cite{Faou-geometric-2012}. 
However, in \cite{schratz-low-2018} the authors show that  in certain applications  splitting methods   suffer from severe order reduction such as  in case of  non-linearities with non-integer exponents. The latter arises for instance in context of optical dark and power law solitons with surface plasmonic interactions \cite{crutcher_derivation_2011}. As a solution to that issue, the authors proposed in  \cite{schratz-low-2018} a new class of low regularity exponential-type integrators for NLS. 
In this article we use a different approach based on the so-called  Scalar Auxiliary Variable (SAV) method which was originally proposed   to design structure-preserving numerical schemes for gradient flows \cite{shen_2018_sav,shen_new_2019}. Very recently it also became popular in context of  Hamiltonian systems  \cite{antoine-SAV-2020,fu-structure-2019,cai-structure-2019,Feng}. The main advantage of the SAV method lies in the fact that it preserves a modified Hamiltonian on the discrete level. Due to its generality, it can be applied to a large class of equations involving any kind of nonlinearity. The resulting numerical schemes are linearly implicit and allow for efficient calculations.

{
The main idea behind the SAV method is to introduce a scalar variable $r(t)=\sqrt{\cae_1+ \cae_c}$ that will become an unknown at the discrete level and where the arbitrary constant $\cae_c >0$ is used to obtain $\cae_1+ \cae_c>0$. We must stress that one as to be very careful with the choice of the constant $\cae_c$. Indeed, it is well known that even for the cubic non-linearity, \textit{i.e.} $f(\abs{p}^2,\abs{q}^2) = \beta \abs{p^2+q^2}^2$ with $\beta<0$ (focussing NLSE), the hamiltonian energy~\eqref{eq:hamiltonian-energy} is not bounded from below a priori. In the following analysis, we implicit assume that it exists a constant $\cae_c$ such that $\cae_1+ \cae_c>0$, which is often the case in the study of Bose-Einstein condensate as pointed out by Antoine \textit{et al.} \cite{antoine-SAV-2020}. In practice, we compute the term $\cae_1$ explicitly and therefore one can adapt the constant $\cae_c$ during the simulation. 
}
The system is supplemented by an equation describing the time evolution of $r(t)$. 
In case of the nonlinear Schr\"odinger equation \eqref{eq:schro} the continuous SAV model takes the form
\begin{equation}
    \begin{cases}
        \p_t p &= -\Delta q + r(t) g_1(p,q),\\
        \p_t q &= \Delta p - r(t)g_2(p,q), \\
        \p_t r(t)&= \frac{1}{2} \left[\scal{g_1(p,q)}{\p_t q} + \scal{g_2(p,q)}{\p_t p} \right],
    \end{cases}
    \label{eq:SAV-schro}
\end{equation}
where $\scal{\cdot}{\cdot}$ denotes the standard $L^2$ scalar product and
\[
    g_1(p,q) = \frac{1}{\sqrt{\cae_1[t]+\cae_c }} \frac{\delta \cae_1[t]}{\delta q}, \quad   g_2(p,q) = \frac{1}{\sqrt{\cae_1[t]+\cae_c}} \frac{\delta \cae_1[t]}{\delta p}.
\]
Associated to this SAV model we find the Hamiltonian 
\begin{equation*}
    \tilde{H}(p,q) = \f12\int_\Omega \abs{\nabla p}^2+ \abs{\nabla q}^2 \dd x + \abs{r}^2,
\end{equation*}
which is conserved by the SAV model \eqref{eq:SAV-schro}. In the following, we assume that for $i=1,2$ 
\begin{equation}
    \begin{aligned}
    \abs{g_i'(p,q)} \le C\left((\abs{p}+\abs{q})^\beta +1\right),\quad
    \abs{g_i''(p,q)} \le C\left((\abs{p}+\abs{q})^{\beta'} +1\right),
    \end{aligned}
    \label{eq:assum-bound-g'}
\end{equation}
for some $\beta,\beta' >0$.

\begin{remark}
In this paper, we focus on the Gross-pitaevskii equation under the form \eqref{eq:schro}.
Even though the choice of the nonlinearity and, therefore, the precise form of $\cae_1$, depends on the structure of the considered Schr\"odinger equation, we highlight that the SAV scheme is, in its design, general enough to work for a large number of applications. Indeed, as long  as there exists a constant $\cae_c$ such that for all times $t\ge 0$ it holds that $\cae_1 + \cae_c >0$ we can apply the SAV method. Therefore, modifications such as the effect of dipole-dipole interactions, rotating GPE (see Antoine \textit{et al.}~\cite{antoine-SAV-2020}), or even time dependent potentials $V = V(t,x)$ can be taken into account. For an extensive overview on applications and generalisations of the nonlinear Schr\"odinger equation, we refer the  interested reader to the review article of Bao and Cai~\cite{Bao-Cai} and the references therein. 
\end{remark}

Following the works of Antoine \textit{et al.} \cite{antoine-SAV-2020} and Fu \textit{et al.} \cite{fu-structure-2019}, we  analyze a fully discrete SAV scheme for the nonlinear Schr\"odinger equation \eqref{eq:schro}  based on a   Crank-Nicholson time discretization of the NLS SAV model \eqref{eq:SAV-schro} coupled with a pseudo-spectral discretization for the spatial discretization.  Energy conservation properties of the SAV method for nonlinear Schr\"odinger equations were recently derived in  \cite{antoine-SAV-2020,fu-structure-2019} and their  convergence was   extensively tested numerically. 
{ {Very recently}, Feng \textit{et al.}~\cite{Feng} use the SAV method to design arbitrary high order space-time finite element scheme for the nonlinear Schr\"odinger equation. While their method uses a finite element discretization in space, we propose in this work to use a Fourier pseudospectral discretization. 
}
The  main contribution of this article lies in establishing global error estimates on the fully discrete Fourier-PseudoSpectral Crank-Nicholson NLS SAV scheme (CN-SAV-SP in short).  More precisely, we derive weak and strong convergence and prove  second order error estimates for the fully discrete scheme.   Our theoretical convergence analysis is inspired by the analysis of the SAV method in the context of gradient flows \cite{shen_2018_convergence}. We   underline our convergence results with numerical experiments and  compare  the SAV scheme with classical splitting methods. 
Our numerical findings suggest that in certain cases, such as in case of non-linearities involving a non-integer exponent, the SAV scheme preserves its second order energy conservation property while classical splitting methods suffer from sever order reduction. We also conduct numerical experiments showing that the SAV scheme is able to compute correctly ground states of Bose-Einstein condensates.\vskip0.2cm 

{\bf Outline of the paper.}
 In the first part of the paper, we carry out a fully discrete error analysis of the SAV scheme and establish second order convergence estimates, see Theorem \ref{thm:main}. Our theoretical convergence results are then numerically underlined in the second part of the paper, see Section \ref{sec:num}.  \vskip0.2cm
 
 {\bf Notations.}
 Let $L^p(\Omega)$, $W^{m,p}(\Omega)$ with $H^m(\Omega) = W^{m,2}(\Omega)$, where $1 \le p \le +\infty$ and $m \in  \mathbb{N}$, denote the standard Lebesgue and Sobolev spaces equipped with the corresponding  norms  $||\cdot||_{m,p}$, $||\cdot||_{m}$ and semi-norms $|\cdot|_{m,p}$, $|\cdot|_{m}$. We also denote by $H^m_p(\Omega)$ the subset of $H^m(\Omega)$ that consists of $2\pi$-periodic functions that are in $H^m(\Omega)$.  
We denote by $L^p\left(0,T;V\right)$ the Bochner spaces i.e. the spaces with values in Sobolev spaces \cite{adams-sobolev-1975}. The norm in these spaces is defined for all Bochner measurable functions $\eta$ by
\[
    \norm{\eta}_{L^p(0,T;V)} = \left(\int_0^T \norm{\eta}^p_V \, \dd t\right)^{1/p}, \qquad
    \norm{\eta}_{L^\infty(0,T;V)} = \text{ess} \sup_{t\in (0,T)} \norm{\eta}_V. 
\]
The standard $L^2$ inner product is denoted by $(\cdot,\cdot)_\Omega$ and the duality pairing between ${\left(H^1(\Omega) \right)' = H^{-1}(\Omega)}$ and $H^1(\Omega)$ by $<\cdot,\cdot>_\Omega$. The dual space $H^{-1}(\Omega)$ is endowed with the norm 
\[
    \norm{\phi}_{H^{-1}(\Omega)} = \sup_{\eta\in H^1(\Omega)}\{ <\phi,\eta>_\Omega,\quad \norm{\eta}^2_{1} \le 1  \}.
\]

\begin{remark}
Even though our model problem~\eqref{eq:schro} is equipped with periodic boundary conditions, our analysis holds for homogeneous Dirichlet or Neumann boundary conditions.
\end{remark}

\section{Numerical scheme}
\label{sec:schemes}

\subsection{Time and space discretisation of the SAV model}

{
We use a standard Fourier pseudospectral method~\cite{Quarteroni-spectral-2006, Fornberg, Gottlieb} for the spatial  discretization of the SAV model  \eqref{eq:SAV-schro}. We refer the reader to the book of Trefethen~\cite{Trefethen} for details of the implementation of such scheme in MATLAB.
We emphasize that our paper presents numerical simulation in dimension $d=1$. However, the method can  be adapted to higher dimensions. Our convergence and error analysis holds in  dimensions $1\le d\le3$.
}

Thereby, for the sake of clarity, we here give the details of the space discretization for $d=1$. We denote by  $X_N$ the space spanned by the trigonometric functions up to degree $N/2$ 
\begin{equation*}
    X_N := \text{span}\{e^{ikx/L}: -N/2\le k \le N/2-1 \}.
\end{equation*}
For the time discretisation of the SAV system  \eqref{eq:SAV-schro}  we   apply a Crank-Nicholson discretisation  with  time step $\tau$ such that $t^k = k\tau$ for $k\in \N$. 
At each grid point we thereby approximate the time derivative by 
\[
    \p_t u(t^{k+1},x) \approx \f{u(t^{k+1},x) - u(t^{k},x)}{\tau}. 
\]

Let us give the details of the approximation in dimension $d=1$, where the domain is defined by $\Omega = [-\pi, \pi]$ with a mesh size $h$. In this case the collocation points are $x_a = \f{2\pi a}{N}$ where $a\in \mathcal B$ with 
\[
    \mathcal{B} := \begin{cases}
            \{-P,\dots,P-1\}\quad \text{if}\quad N=2P \quad \text{is even},\\
            \{-P,\dots,P\}\quad \text{if}\quad N=2P+1 \quad \text{is odd}.
    \end{cases}  
\]
We denote by $U^k(x_a)$ the approximation of $u(t^k,x_a)$. The Fourier pseudo-spectral discretization is given by
\[
    U^k(x_a) = \sum_{p\in \mathcal B} \hat{u}_p^k \exp\left(2i \pi a p / N\right)
\]
with the Fourier coefficients  defined by 
\[
    \hat{u}_p^k = \f{1}{N} \sum_{b\in \mathcal B} U^k(x_b) \exp\left(- 2i \pi b p / N\right). 
\]

We approximate the Laplacian by the Fourier differentiation matrix $D^{(2)}$ which  for $j,l=0,\dots,N-1$ takes the form  
\[
    \left(D^{(2)}\right)_{jl} = \begin{cases}
        \f14 (-1)^{j+1}N+\f{(-1)^{j+l+1}}{2\sin^2\left( \f{(j-l)\pi}{N}\right)}, \quad &\text{if}\quad j\neq l\\
        -\f{(N-1)(N-2)}{12}, \quad &\text{otherwise}.
    \end{cases}
\]

\AP{
However, to avoid the need to use the symmetric matrix $D^{(2)}$ in the previous form and gain in computational time, it can be preferable to use the method proposed in \cite{antoine-SAV-2020} that uses the fact that the previous differentiation matrix is diagonal in Fourier space. Therefore, inverting this matrix has a very low cost. In our work, since we use the Hamiltonian system~\eqref{eq:SAV-ham} to analyze the properties of the SAV scheme we will use the previously defined differentiation matrix $D^{(2)}$.   
}


For the $N$ collocation points $x_a$, we define the interpolation operation $I_N$  by
\begin{equation*}
    (I_N u)(x) = \sum_{p\in \mathcal{B}} \tilde u_p e^{2 i \pi x p/N }.
\end{equation*}
We have the following interpolation error (see Section 5.8.1 in \cite{Quarteroni-spectral-2006}):
\begin{lemma}[Interpolation error]
    \label{lem:interp-error}
    For any $u\in C(0,T;H^m_p(\Omega)$ with $d\le3$, we have
    \begin{equation*}
        \begin{cases}
            \norm{I_N u - u }_{H^l_p(\Omega))}\le C N^{l-m}\abs{u}_m,\quad 0\le l \le m,\\
            \norm{I_N \p_t u - \p_t u }_{H^l_p(\Omega))}\le C N^{l-m}\abs{\p_t u}_m,\quad 0\le l \le m.
        \end{cases}
    \end{equation*}
\end{lemma}

\subsection{The fully discrete SAV scheme}

Applying the time discretization described in the previous section, for $k=0\to N_T$, the semi-discrete model of \eqref{eq:SAV-schro} reads
\begin{equation}
    \begin{cases}
        \frac{p^{k+1}-p^k}{\tau} &= -\Delta q^{k+1/2} + r^{k+1/2}\tilde g_1^{k+1/2}, \\
        \frac{q^{k+1}-q^k}{\tau} &= \Delta p^{k+1/2} - r^{k+1/2} \tilde g_2^{k+1/2}, \\
        {r^{k+1}-r^k} &= \frac{1}{2} \left[\scal{\tilde g_1^{k+1/2}}{q^{k+1}-q^k} + \scal{\tilde g_2^{k+1/2}}{p^{k+1}-p^k} \right],
    \end{cases}
    \label{eq:SAV-schro-discr}
\end{equation}
where $\phi^{k+1/2} = \left(\phi^{k+1}+\phi^k\right)/2$ and $\tilde g_i^{k+1/2}$ is a second order extrapolation of $g_i$ at time $t=t^{k+1/2}$.

Denoting by capital letters the vectors of unknowns $P^k,Q^k$ that are approximations at each collocation nodes of the continuous (in space) unknowns $p^k,q^k$, the fully discrete space-time scheme then takes the form
\begin{equation}
    \begin{cases}
        \frac{P^{k+1}-P^k}{\tau} &= -D^{(2)} Q^{k+1/2} + R^{k+1/2}\tilde G_1^{k+1/2}, \\
        \frac{Q^{k+1}-Q^k}{\tau} &= D^{(2)} P^{k+1/2} - R^{k+1/2} \tilde G_2^{k+1/2}, \\
        {r^{k+1}-r^k} &= \frac{1}{2} \left[\scal{\tilde G_1^{k+1/2}}{Q^{k+1}-Q^k} + \scal{\tilde G_2^{k+1/2}}{P^{k+1}-P^k} \right],
    \end{cases}
    \label{eq:SAV-schro-fulldiscr}
\end{equation}
where $\tilde G_1, \tilde G_2$ are the vectors associated to the functions $\tilde g_1$ and $\tilde g_2$.

Let us now present two algorithms for the efficient solution of  the fully discrete SAV system \eqref{eq:SAV-schro-fulldiscr}. 
The two methods are equivalent and reduce the problem to the solving of two linear systems that involves only real variables.
\subsubsection*{Algorithm 1}
 The algorithm below was originally proposed for solving  the fully discrete SAV system arising in gradient flows \cite{shen_2018_sav}.
 
Let us give the procedure on how to solve the system \eqref{eq:SAV-schro-fulldiscr}. First, we need to replace $r^{k+1}$ in the first two equations using the third equation. This yields that 
\[
    \begin{cases}
        \left(P^\kk-P^k \right) &= -\tau D^{(2)} \f{(Q^{k+1}+Q^k)}{2} + \tau\left(r^k + \f14 \left[\scal{\tilde{G}_1^{k+1/2}}{Q^{k+1}-Q^k}+\scal{\tilde G^{k+1/2}_2}{P^\kk-P^k} \right] \right)\tilde G_1^{k+1/2},\\
        \left(Q^\kk-Q^k\right) &= \tau D^{(2)}\f{\left(P^\kk-P^k \right)}{2} - \tau \left(r^k + \f14 \left[\scal{\tilde{G}_1^{k+1/2}}{Q^{k+1}-Q^k}+\scal{\tilde G^{k+1/2}_2}{P^\kk-P^k} \right] \right)\tilde G_2^{k+1/2}.
    \end{cases}
\]
Next we set
\[
    Z^k = \begin{pmatrix}
        P^k\\
        Q^k
    \end{pmatrix}, \tilde{G}^{k+1/2} = \begin{pmatrix}
        \tilde{G}_2^{k+1/2}\\
        \tilde{G}_1^{k+1/2}
    \end{pmatrix}  , \tilde{B}^{k+1/2} = \begin{pmatrix}
        -\tilde{G}_1^{k+1/2}\\
        \tilde{G}_2^{k+1/2}
    \end{pmatrix}.
\]
This allows us to rewrite the system into a matrix form 
\begin{equation}
    A Z^{k+1} + \f\tau4\scal{\tilde{G}^{k+1/2}}{Z^{k+1}}\tilde{B}^{k+1/2} = C^k, 
    \label{eq:discrete-matrix-form}
\end{equation}
where 
\[
    A = \begin{bmatrix}
        I & \f\tau2 D^{(2)}\\
        -\f{\tau}{2}D^{(2)} & I 
    \end{bmatrix}, \quad \text{and} \quad C^k=  \begin{pmatrix}
        I&-\f\tau2 D^{(2)} \\
        \f\tau2 D^{(2)}&I
    \end{pmatrix}Z^k-\tau r^k\tilde{B}^{k+1/2} + \f\tau4 \scal{\tilde{G}^{k+1/2}}{Z^k}\tilde{B}^{k+1/2},
\]
with $I$ the identity matrix. 
Multiplying \eqref{eq:discrete-matrix-form} by $A^{-1}$ and taking the discrete inner product with $\tilde G^{k+1/2}$, we finally obtain 
\begin{equation}
    \scal{\tilde{G}^{k+1/2}}{Z^{k+1}} = \f{\scal{\tilde G^{k+1/2}}{A^{-1}C^k}}{1+\f\tau4 \scal{\tilde G^{k+1/2}}{A^{-1}\tilde B^{k+1/2}}}.
    \label{eq:scal-g-sav}
\end{equation}   
Then, knowing $\scal{\tilde{G}^{k+1/2}}{Z^{k+1}}$, $Z^{k+1}$ is computed using \eqref{eq:discrete-matrix-form} and $r^{k+1}$ is calculated from the third equation of \eqref{eq:SAV-schro-fulldiscr}.
Therefore, solving the fully discrete SAV model \eqref{eq:SAV-schro-fulldiscr} reduces to solving the linear system constituted by the equations \eqref{eq:scal-g-sav} and \eqref{eq:discrete-matrix-form}. 

\subsubsection*{Algorithm 2}
Below we describe a second algorithm recently proposed in  \cite{antoine-SAV-2020} for the numerical solution of the fully discrete NLS SAV scheme \eqref{eq:SAV-schro-fulldiscr}. Rewriting the scheme in its matrix form we have 
\begin{equation}
    \begin{cases}
        \f{Z^{k+1}-Z^k}{\tau} &= -J Z^{k+1/2} - r^{k+1/2} \tilde{B}^{k+1/2},\\
        {r^{k+1}-r^k} &= \frac{1}{2} \left[\scal{\tilde G_1^{k+1/2}}{Q^{k+1}-Q^k} + \scal{\tilde G_2^{k+1/2}}{P^{k+1}-P^k} \right],
    \end{cases}
    \label{eq:discr-system-compact}
\end{equation}
where 
\[
    J=\begin{bmatrix}
        0&D^{(2)}\\ -D^{(2)}&0
    \end{bmatrix}  .
\]
Using the decomposition 
\begin{equation}
    Z^{k+1/2} =  Z^{k+1/2}_1 + r^{k+1/2} Z^{k+1/2}_2,   
    \label{eq:decomp}
\end{equation}
and adding $\f2\tau Z^k$ on both sides of the first equation of \eqref{eq:discr-system-compact}, we furthermore obtain that
\begin{equation}
    \f2\tau\left[Z^{\kud}+r^\kud Z^\kud_2\right] = \f2\tau Z^k - J \left[Z_1^\kud+r^\kud Z_2^\kud \right]-r^\kud\tilde B^\kud.
    \label{eq:decomp-sys}
\end{equation}
Applying the same decomposition to the second equation of \eqref{eq:discr-system-compact} and adding $2r^k$ on both sides, we get 
\begin{equation}
    2r^\kud = 2r^k + \scal{\tilde{G}^\kud}{\left[Z_1^\kud+r^\kud Z_2^\kud\right]-Z^k}.
    \label{eq:equation-r-decomp}
\end{equation}
    
Hence, denoting by $I$ the identity matrix, we first solve the equation \eqref{eq:decomp-sys} using the system 
\begin{equation}
\begin{cases}
    \left[\f2\tau I +J\right]Z_1^\kud = \f2\tau Z^k\\
    \left[\f2\tau I +J\right]Z_2^\kud = - \tilde B^\kud.
\end{cases}    
\label{eq:discrete-algo-2}
\end{equation}
Then we compute $r^\kud$ by solving  equation \eqref{eq:equation-r-decomp} which yields that
\begin{equation*}
    r^\kud = \f{2 r^k+\scal{\tilde G^\kud}{Z^\kud_1-Z^k}}{2-\scal{\tilde{G}^\kud}{Z_2^\kud}}.
\end{equation*}
From the decomposition \eqref{eq:decomp} we get $Z^\kud$ from which we compute $Z^{k+1}$ and $r^{k+1}$.

\AP{
\begin{remark}
    Since in practice the computation and storage of the invert of a non-diagonal matrix has to be avoided, Algorithm 2 is a preferable choice. Indeed, the main step in Algorithm 2 lies in  solving   two decoupled linear equations~\eqref{eq:discrete-algo-2}. To do so, standard tools of linear systems can be applied such as matrix-free preconditioned Krylov solvers. We refer to Appendix C in \cite{Shen-spectral} for a description of iterative solvers of linear system and preconditioning.
\end{remark}
Even though the previous remark already highlights the main advantage of Algorithm 2, we emphasize that the inversion of the main matrix in Algorithm 1 can be carried out efficiently. 
\begin{remark}
    Referring to \cite{shen_2018_sav}, we remark that the inversion of the matrix $A$ in the first algorithm and the matrix  $ \left[\f2\tau +J\right]$ in the second Algorithm can be carried out efficiently using the Sherman-Morrison-Woodbury formula \cite{Golub-matrix-2013}   
    \[
        (A+UV^{T})^{-1} = A^{-1}-A^{-1}U(I+V^TA^{-1}U)^{-1}V^TA^{-1},   
    \]
    where $A$ is a $n\times n$ and $U,V$ are $n\times k$ matrices, and $I$ is the $k\times k$ identity matrix. 
\end{remark}
\begin{remark}
    Referring to \cite{fu-structure-2019}, a fast solver for solving  the linear system~\eqref{eq:discrete-matrix-form} (in Algorithm~1) and~\eqref{eq:discrete-algo-2} (in Algorithm 2) exists. It uses the fact that the differentiation matrix $D^{(2)}$ can be decomposed into $D^{(2)} = F^{-1} \Lambda F$ where $F$ and $F^{-1}$ are the corresponding matrices for the discrete Fourier transformation and $\Lambda$ is a diagonal matrix with eigenvalues of $D^{(2)}$ as its entries. 
    Therefore, the matrix $A$ from Equation~\eqref{eq:discrete-matrix-form} admits the decomposition 
    \[
        A = F^{-1} M F, \quad\text{with}\quad M = \begin{bmatrix} I & \f{\tau}{2} \Lambda\\ - \f{\tau}{2} \Lambda & I \end{bmatrix}.
    \]
    We note that a similar decomposition exists for the matrix $[\f{2}{\tau}I +J]$ in equation~\eqref{eq:discrete-algo-2}. Thanks to the above decomposition, the inverse of the matrix $A$ can be computed explicitly in an efficient manner since
    \[
        A^{-1} = F^{-1} M^{-1} F,\quad \text{and}\quad M^{-1} = M^T \begin{bmatrix}(I+\f{\tau^2}{4}\Lambda^2)^{-1}& 0 \\
        0 & (I+\f{\tau^2}{4}\Lambda^2)^{-1}\end{bmatrix},
    \]
    where $(I+\f{\tau^2}{4}\Lambda^2)$ is a diagonal matrix, such that its inverse is fast to compute. 
\end{remark}
}

\section{Conservation properties and inequalities}
\label{sec:apriori}
\AP{
In this section we outline the  conserved quantities of the SAV method. It is well known that due to its design the SAV scheme preserves a modified version of the underlying Hamiltonian. In addition, to the conservation of energy, there is a wide variety of properties in the continuous equation which is feasible to preserve also on the numerical (discrete) level, we refer to  Bao and Cai~\cite{Bao-Cai} as well as Antoine \textit{et al.}~\cite{Antoine-Computational}: \textit{i)}  \textit{time-reversibility or symmetry}, \textit{i.e.} the system is unchanged when $\tau\to -\tau$, \textit{ii)}  \textit{gauge-invariance}, \textit{i.e.} if the potential $V$ is changed such that $V\to V+\alpha$ with $\alpha$ a real constant then the density $\abs{u}^2$ remains unchanged, \textit{iii)} conservation of mass, \textit{i.e.} $\norm{u(t)}_{L^2(\Omega)} = \norm{u(0)}_{L^2(\Omega)}$, and the Hamiltonian energy, \textit{i.e.} $H(t) = H(0)$, \textit{iv)} preservation of the dispersion relation
\[
    \omega(k) = \frac{\abs{k}^2}{2} + f(\abs{A}^2),
\]
for the plane wave solutions $u(t,x) = Ae^{i k\cdot x-\omega t}$.
}

\AP{
Proving analytically that the SAV scheme for the NLS equation meets the points $ii)$ and $iv)$ (over long time scales) is up to our knowledge not possible with current techniques. However, the other points can be verified for a large number of nonlinearities. 
Here, we briefly recall the proofs of the conservation  properties  and refer to \cite{fu-structure-2019,antoine-SAV-2020}, where they have been first set in context of nonlinear Schr\"odinger equations and our Theorems~\ref{th:preservation-ham} and~\ref{th:mass-cons} are found by a combination of the results from~\cite{fu-structure-2019} and~\cite{antoine-SAV-2020}.
}
\begin{theorem}[Conservation of the modified discrete energy]
    \label{th:preservation-ham}
    The scheme \eqref{eq:SAV-schro-discr} is associated to the discrete modified Hamiltonian
    \begin{equation}
        \tilde{H}^{k+1}= \f12\left(\abs{ Q^{k+1}}^2_1  + \abs{ P^{k+1}}^2_1 \right) + \abs{r^{k+1}}^2, 
        \label{eq:discrete-ham}
    \end{equation}
    and conserves the modified Hamiltonian energy through time i.e.
    \begin{equation}
        \tilde{H}^{k+1} = \tilde{H}^k.
        \label{eq:pres-ham}
    \end{equation}
\end{theorem}
\begin{proof}
    Taking the inner product with $Q^{k+1}-Q^k$ for the first equation of \eqref{eq:SAV-schro-fulldiscr} and for the second with $-\left(P^{k+1}-P^k\right)$, then summing the results we get
    \[
        0 = \f12\left(\abs{Q^{k+1}}^2_1 - \abs{Q^k}^2_1 + \abs{P^{k+1}}^2_1 - \abs{P^k}^2_1 \right) +r^{k+1/2} \left[\scal{\tilde G_1^{k+1/2}}{Q^{k+1}-Q^k} + \scal{\tilde G_2^{k+1/2}}{P^{k+1}-P^k} \right],
    \]   
    where $\abs{\cdot}_1 = \norm{\nabla \cdot}_0$ is the $H^1-$seminorm. 
    Then, multiplying the third equation of \eqref{eq:SAV-schro-fulldiscr} by $2 R^{k+1/2}$ and using the result in the previous equation, we obtain 
    \[
        0 = \f12\left(\abs{ Q^{k+1}}^2_1 - \abs{ Q^k}^2_1 + \abs{ P^{k+1}}^2_1 - \abs{ P^k}^2_1 \right) +\left( \abs{r^{k+1}}^2 - \abs{r^k}^2\right),
    \]  
    from which we can conclude both \eqref{eq:discrete-ham} and \eqref{eq:pres-ham}.
    \hfill \qed
\end{proof}

The SAV scheme also preserves the mass up to an error of order $ \mathcal{O}(\tau^3)$, where the latter error is introduced by the second-order extrapolation.
\begin{theorem}[Conservation of the $L^2$ norm]
    The scheme \eqref{eq:SAV-schro} conserves the $L^2$ norm of the solution up to an order $ \mathcal{O}(\tau^3)$ i.e.
    \begin{equation}
        \norm{U^{k+1}}^2_{0} = \norm{U^{k}}^2_{0} + \mathcal{O}(\tau^3),
        \label{eq:conservation-mass}
    \end{equation}
    with $U^k = P^k +i Q^k$.
    \label{th:mass-cons}
\end{theorem}
\begin{proof}
    Taking the inner product of first equation of \eqref{eq:SAV-schro-fulldiscr} with $2P^{k+1/2}$, the second equation with $2Q^{k+1/2}$, and summing the two we get 
    \begin{equation*}
        \f{1}{\tau} \left(\norm{P^{k+1}}_0^2 - \norm{P^{k}}_0^2 + \norm{Q^{k+1}}_0^2 - \norm{Q^{k}}_0^2\right) = 2 r^{k+1/2}\left(-\scal{\tilde G_2^{k+1/2}}{Q^{k+1/2}} +\scal{\tilde G_1^{k+1/2}}{P^{k+1/2}}  \right).   
    \end{equation*}
    
    Since $\tilde G_i^{k+1/2}$ is a second-order approximation of $G_i^{k+1/2}$, we can write
    \begin{equation*}
        \f 1{\tau} \left(\norm{U^{k+1}}_0^2 - \norm{U^{k}}_0^2 \right) = 2 r^{k+1/2}\left(-\scal{G_2^{k+1/2}}{Q^{k+1/2}}     +\scal{ G_1^{k+1/2}}{P^{k+1/2}}  \right) + \mathcal{O}(\tau^2).
    \end{equation*}
    \AP{
    Then, we find that
    \[
    \begin{aligned}
        G_1^{k+1/2} &= \f{1}{\sqrt{\cae_1\left(P^{k+1/2},Q^{k+1/2} \right) + \cae_c}} \f{\p \cae_1\left(P^{k+1/2},Q^{k+1/2}\right)}{\p Q^{k+1/2}} \\
        &= V(x) Q^{k+1/2} + f\left(\abs{P^{k+1/2}}^2,\abs{Q^{k+1/2}}^2\right)Q^{k+1/2},
    \end{aligned}
    \]
    and
    \[
    \begin{aligned}
        G_2^{k+1/2} &= \f{1}{\sqrt{\cae_1\left(P^{k+1/2},Q^{k+1/2} \right) + \cae_c}} \f{\p \cae_1\left(P^{k+1/2},Q^{k+1/2}\right)}{\p P^{k+1/2}} \\
        &= V(x) P^{k+1/2} + f\left(\abs{P^{k+1/2}}^2,\abs{Q^{k+1/2}}^2\right)P^{k+1/2},
    \end{aligned}
    \]
    from which we easily obtain
    \[
        -\scal{G_2^{k+1/2}}{Q^{k+1/2}}    +\scal{G_1^{k+1/2}}{P^{k+1/2}}=0.    
    \]
    Consequently, we obtain \eqref{eq:conservation-mass}.
    }
    \hfill \qed
\end{proof}
\AP{
To derive $H^2$-bound for the solution of the SAV scheme, we use the following proposition. The proof of this technical result can be found in Lemma 2.3 in \cite{shen_2018_convergence}. 
\begin{proposition}[Bound for $\norm{\nabla G_i^{k+1/2}}_0$]
\label{prop:bound-nablag}
    Assume that the functions $g_i$ (i=1,2) satisfy \eqref{eq:assum-bound-g'} and let  $\norm{U}_1 \le {M}$ for some constant $M>0$. Then there exists $0\le \sigma<1$ such that 
    \begin{equation}
        \norm{ \nabla G_i^{k+1/2}} \le C(M)\left(1+\norm{\nabla\Delta P^{k+1/2}}_0^{2\sigma}+\norm{\nabla\Delta Q^{k+1/2}}_0^{2\sigma}\right).
    \end{equation}
\end{proposition}
}

\AP{
We have the following result on the $H^2$-norm of $P^{k+1/2}$ and $Q^{k+1/2}$.
\begin{proposition}[$H^2$-bound on the numerical solution]
\label{prop:H2bound}
    The solution $\{P^\kk,Q^\kk\}$ of \eqref{eq:SAV-schro-discr} satisfies 
    \begin{equation}
        \max_{k=1,\dots,N_T-1} \norm{\Delta P^\kk}_{0}^2 + \norm{\Delta Q^\kk}_{0}^2 \le CT + \norm{\Delta P^0}_{0}^2 + \norm{\Delta Q^0}_{0}^2 .
        \label{eq:H2-estim-num}
    \end{equation}
\end{proposition}
\begin{proof}
    First, we multiply the first equation of \eqref{eq:SAV-schro-discr} by $\Delta^2 (Q^{k+1/2})$, the second equation by $\Delta^2 (P^{k+1/2})$ and integrate over $\Omega$. Then, by summing the two, we obtain,  after integration by parts, that
    \[
        \begin{aligned}
          \norm{\nabla \Delta Q^{k+1/2}}^2_0 +  \norm{\nabla \Delta P^{k+1/2}}^2_0 =     \scal{r^{k+1/2}  \nabla \tilde G_1^{k+\f12}}{\nabla \Delta Q^{k+1/2})} +  \scal{r^{k+1/2} \nabla \tilde G_2^{k+\f12}}{\nabla \Delta P^{k+1/2}}.
        \end{aligned}
    \]
    From the conservation of the modified Hamiltonian~\eqref{eq:discrete-ham}--\eqref{eq:pres-ham} and assuming a finite initial Hamiltonian, we have 
    \[
    \begin{aligned}
     &\scal{r^{k+1/2}  \nabla \tilde G_1^{k+\f12}}{\nabla \Delta Q^{k+1/2}} +  \scal{r^{k+1/2} \nabla \tilde G_2^{k+\f12}}{\nabla \Delta P^{k+1/2}} \\
     &\le \frac{C}{2} \left( \norm{\nabla \tilde G_1^{k+\f12}}_0^2 + \norm{\nabla \Delta Q^{k+1/2}}_0^2 + \norm{\nabla \tilde G_2^{k+\f12}}^2_0  + \norm{\nabla \Delta P^{k+1/2}}^2_0 \right).
     \end{aligned}
    \]
    Then, from the result of Proposition~\ref{prop:bound-nablag}, for any $\epsilon>0$, we have 
    \[
        \norm{\nabla \tilde G_1^{k+\f12}}_0^2 + \norm{\nabla \tilde G_2^{k+\f12}}^2_0  \le \epsilon  \norm{\nabla \Delta Q^{k+1/2}}_0^2 + \epsilon\norm{\nabla \Delta P^{k+1/2}}^2_0   + C(\epsilon).
    \]
    Therefore, combining the two previous inequalities, we obtain 
    \begin{equation}
        \norm{\nabla \Delta Q^{k+1/2}}^2_0 +  \norm{\nabla \Delta P^{k+1/2}}^2_0 \le C.
        \label{eq:boundnabladelta}
    \end{equation}
    Secondly, by multiplying the first equation of \eqref{eq:SAV-schro-discr} with $\Delta^2 (P^{k+1/2})$, the second equation with $\Delta^2 (Q^{k+1/2})$, integrating over $\Omega$, and summing the two, we obtain after integration by parts that
    \[
        \begin{aligned}
        \norm{\Delta P^\kk}_0^2&-\norm{\Delta P^k}^2_0 + \norm{\Delta Q^\kk}_0^2-\norm{\Delta Q^k}^2_0  \\
        &=  \tau r^{k+\f12} \scal{\nabla \tilde G_1^{k+\f12}}{\nabla \Delta Q^{k+1/2}} - \tau r^{k+\f12} \scal{\nabla \tilde G_2^{k+\f12}}{\nabla \Delta P^{k+1/2}}.
        \end{aligned}
    \]
    Then, combining the result of Proposition~\ref{prop:bound-nablag} and the inequality \eqref{eq:boundnabladelta}, we have
    \[
        \norm{\Delta P^\kk}_0^2-\norm{\Delta P^k}^2_0 + \norm{\Delta Q^\kk}_0^2-\norm{\Delta Q^k}^2_0  \le \tau C,
    \]
     and summing from $k=0\to N_T$, we obtain \eqref{eq:H2-estim-num}.
    \hfill \qed
\end{proof}
}

\AP{
\begin{remark}
From the fact that $H^2(\Omega) \subseteq L^\infty(\Omega) $ for $d\le 3$, we can conclude from the previous proposition that for $k=1,\dots N_T-1$, 
\begin{equation}
    \norm{P^{k+1}}_{L^\infty} + \norm{Q^{k+1}}_{L^\infty} \le C .
    \label{eq:linf-bound}
\end{equation}
\label{rem:linf}
\end{remark}
}

Next, we present the stability inequality that will be useful in the convergence analysis.  
\begin{proposition}[Stability inequality]
    The solution of \eqref{eq:SAV-schro-discr} satisfies the stability inequality
    \begin{equation}
        \max_{k=0,\dots,N_T-1}\left[ \norm{P^\kk}_0^2+ \norm{Q^\kk}_0^2\right] + \tau^2 \sum_{k=0}^{N_T-1}\left[\norm{\frac{P^\kk - P^k}{\tau}}_0^2  +  \norm{\frac{Q^\kk - Q^k}{\tau}}_0^2\right] \le C(\tau, H^0, N_T).
        \label{eq:stab-ineq}
    \end{equation}
    \label{prop:H2-bound}
\end{proposition}
\begin{proof}
    Multiplying the first equation with $2\tau P^{k+1}$, integrating over $\Omega$ and using $2(a-b)a = a^2-b^2+(a-b)^2$, we obtain 
    \[
        \norm{P^\kk}_0^2 + \tau^2 \norm{\frac{P^\kk - P^k}{\tau}}_0^2 - \norm{P^k}^2_0 = - 2 \tau \scal{\nabla Q^{k+1/2}}{\nabla P^\kk} + 2\tau r^{k+1/2} \scal{\tilde G_1^{k+1/2}}{P^\kk}.
        \label{eq:proof-stab1}
    \]
    \AP{
    Using the Cauchy-Schwartz inequality and \eqref{eq:discrete-ham}--\eqref{eq:pres-ham}, we obtain 
    \[
        - 2 \tau \scal{\nabla Q^{k+1/2}}{\nabla P^\kk} \le 2 \tau \norm{\nabla Q^{k+1/2}}_0 \norm{\nabla P^\kk}_0 \le 4 \tau H^0.
    \]
    Then, from the conservation of the Hamiltonian~\eqref{eq:discrete-ham}--\eqref{eq:pres-ham}, and the conservation of the $L^2$-norm of the solution~\eqref{eq:conservation-mass}, we obtain using the Cauchy-Schwartz inequality
    \[
    r^{k+1/2} \scal{\tilde G_1^{k+1/2}}{P^\kk} \le C \left(\norm{\tilde G_1^{k+1/2}}_0\norm{P^\kk}_0 \right)\le C \norm{ G_1^{k+1/2}}_0 + \mathcal{O}(\tau^2).  
    \]
    Since from Proposition~\ref{prop:H2bound} and \eqref{eq:linf-bound}, we have that $\norm{G_i^{k+1/2}}_0\le C$ with $i=1,2$, for a large number of nonlinearities. 
    Therefore, combining the previous inequalities for the right-hand side of~\eqref{eq:proof-stab1}, we obtain
    \[
        \norm{P^\kk}_0^2 + \tau^2 \norm{\frac{P^\kk - P^k}{\tau}}_0^2 \le C\tau + \norm{P^k}_0^2.
    \]
    The same can be found for the second equation by repeating the same calculations. Summing for $k=0\to N_T-1$, we find \eqref{eq:stab-ineq}.}
    \hfill \qed
\end{proof}

\section{Convergence analysis}
\subsection{Notations}
To study the convergence of the scheme, we introduce the following notation: For $k = 0, \dots, N_T - 1$ we set
\begin{equation*}
	U(t,x) := \frac{t-t^k}{\tau } U^{k+1} + \frac{t^{k+1}-t}{\tau } U^{k}, \quad t \in (t^k, t^{k+1}],
\end{equation*}
and
\begin{equation*}
	\frac{\partial U}{\partial t} := \frac{U^{k+1}-U^k}{\tau } \quad t \in (t^k, t^{k+1}].
\end{equation*}
We also define
\[
    U^+ :=  U^{k+1},\quad U^- := U^k,
\]
and
\begin{equation*}
	U - U^+ = (t-t^{k+1}) \frac{\partial U}{\partial t},\quad U - U^- = (t-t^{k}) \frac{\partial U}{\partial t} \quad t \in (t^k, t^{k+1}], \quad k\ge 0.
\end{equation*}

In addition, we take analogous definitions for $P$ and $Q$: For $k = 0, \dots, K_T - 1$ we set
\begin{equation*}
	P(t,x) := \frac{t-t^k}{\tau} P^{k+1} + \frac{t^{k+1}-t}{\tau} P^{k}, \quad t \in (t^k, t^{k+1}],
\end{equation*}
\begin{equation*}
	\frac{\partial P}{\partial t} := \frac{P^{k+1}-P^k}{\tau} \quad t \in (t^k, t^{k+1}], 
\end{equation*}
\[
    P^+ :=  P^{k+1},\quad P^- := P^k,
\]
and
\begin{equation*}
	P - P^+ = (t-t^{k+1}) \frac{\partial P}{\partial t}, \quad \text{and} \quad P - P^{-} = (t-t^{k}) \frac{\partial P}{\partial t} \quad t \in (t^k, t^{k+1}], \quad k\ge 0.
\end{equation*}
\subsection{Convergence theorem} Now we are in the position to establish time  convergence for the semi discrete SAV scheme \eqref{eq:SAV-schro-discr}.

\begin{theorem}[Convergence]
Let $\{p,q\}$ be a pair of functions such that 
\begin{equation*}
    \begin{cases}
        p(t,x) \in L^2\left([0,T];H^1(\Omega) \right) \bigcap H^1\left([0,T];\left(H^1(\Omega)\right)' \right) \\
        q(t,x) \in L^2\left([0,T];H^1(\Omega) \right) \bigcap H^1\left([0,T];\left(H^1(\Omega)\right)' \right).
    \end{cases}
\end{equation*}
    Then for $\tau \to 0$ we can extract a subsequence of solutions of \eqref{eq:SAV-schro-discr}, such that 
    \begin{align}
        P, P^\pm \rightarrow p\quad \text{strongly in} \quad L^2\left([0,T]; L^2(\Omega)\right),\label{eq:strong-p}\\
        Q, Q^\pm \rightarrow q\quad \text{strongly in} \quad L^2\left([0,T]; L^2(\Omega)\right),\label{eq:strong-q}\\
        P, P^\pm \rightharpoonup p \quad \text{weakly in} \quad L^2\left([0,T]; H^1(\Omega)\right), \label{eq:weak-p}\\ 
        Q, Q^\pm \rightharpoonup q \quad \text{weakly in} \quad L^2\left([0,T]; H^1(\Omega)\right), \label{eq:weak-q}\\
        \frac{\partial P}{\partial t} \rightharpoonup \frac{\partial p}{\partial t} \quad \text{weakly in} \quad L^2\left([0,T]; \left(H^1(\Omega)\right)'\right), \label{eq:weak-dtp}\\
        \frac{\partial Q}{\partial t} \rightharpoonup \frac{\partial q}{\partial t} \quad \text{weakly in} \quad L^2\left([0,T]; \left(H^1(\Omega)\right)'\right), \label{eq:weak-dtq}\\
        r^\kk \rightharpoonup r(t) = \sqrt{\cae_1[t]+\cae_c}\quad \text{weak-star in} \quad L^\infty\left(0,T\right). \label{eq:weakstar-r}
    \end{align}
    The limit $\{p,q\}$ satisfies the nonlinear Schrödinger model \eqref{eq:SAV-ham} in the following weak sense
    \begin{equation}
        \begin{cases}
            \int_0^T \left< \f{\p p}{\p t}, \eta\right>\dd t &=  \int_0^T \int_\Omega \nabla q \nabla \eta + \left(V(x)q + \f{\p F(\abs{p}^2,\abs{q}^2)}{\p q} \right) \eta \dd x \dd t \\
            \int_0^T \left< \f{\p q}{\p t}, \eta\right>\dd t &=  \int_0^T \int_\Omega -\nabla p \nabla \eta -  \left(V(x)p + \f{\p F(\abs{p}^2,\abs{q}^2)}{\p p} \right)\eta \dd x \dd t, \\
        \end{cases}
        \label{eq:limit-sys}
    \end{equation}
    for all $\eta \in L^2\left([0,T]; H^1(\Omega) \right)$.
\end{theorem}
\begin{proof}

    \textit{Step 1: Weak and strong convergences. }
    First, the weak convergences \eqref{eq:weak-p}, \eqref{eq:weak-q}, \eqref{eq:weak-dtp} and \eqref{eq:weak-dtq} follow from the assumption that the initial Hamiltonian energy is bounded and the stability inequality~\eqref{eq:stab-ineq}.

    Then, the weak-star convergence \eqref{eq:weakstar-r} also holds true by the conservation of the modified Hamiltonian and the boundedness of  the initial state. 

    From the compact embedding $H^1(\Omega) \subset L^2(\Omega) \equiv \left(L^2(\Omega)\right)'$, we can apply the Lions-Aubin Lemma~\cite{lions-nonlineaire} to find both convergences \eqref{eq:strong-p} and \eqref{eq:strong-q}.

    \textit{Step 2: Limit system. }
    Let us work on the first equation of the discrete system. We use a test function $\eta \in L^2\left([0,T]; H^1(\Omega) \right)$ and analyze the convergence of the terms separately. First, from the weak convergence \eqref{eq:weak-q}, we have 
    \[
        \int_0^T \int_\Omega \nabla \left(\frac{Q^+ + Q^-}{2} \right)\nabla \eta\dd x \dd t \to \int_0^T \int_\Omega \nabla q \nabla \eta\dd x \dd t.
    \]
    Secondly, from the fact that $\tilde G_1^{k+1/2}$ is a second-order approximation of $G_1^{k+1/2}$, we have 
    \[
        \int_0^T \int_\Omega r^{k+1/2} \tilde G_1^{k+1/2} \eta \dd x \dd t = \int_0^T \int_\Omega r^{k+1/2} G_1^{k+1/2} \eta \dd x \dd t + \int_0^T \int_\Omega r^{k+1/2} \mathcal{O}(\tau^2) \eta \dd x \dd t.
    \]
    From the inequality 
    \[
        \abs{\cae_1(U^{k+1/2}) - \cae_1(u)} \le C \norm{U^{k+1/2}}^2_{L^1(\Omega)}, 
    \]
    and the fact that 
    \[
        P^\pm,Q^\pm \rightharpoonup p,q \quad \text{weak-star in} \quad L^\infty\left(0,T;H^1(\Omega)\right),  
    \]
    which follows from the conservation of both the Hamiltonian energy and the $L^2$ norm, we have 
    \[
        \cae_1(U^{k+1/2})\rightharpoonup \cae_1(u) \quad \text{weak-star in} \quad L^\infty\left(0,T\right).
    \]
    The same holds true for $\f{\delta \cae_1^{k+1/2}}{\delta q}$ and $\f{\delta \cae_1^{k+1/2}}{\delta p}$ using similar arguments. 
    Then, using also the strong convergences \eqref{eq:strong-p} and \eqref{eq:strong-q}, together with the weak-star convergence \eqref{eq:weakstar-r}, we obtain 
    \[
        \int_0^T \int_\Omega r^{k+1/2} G_1^{k+1/2} \eta \dd x \dd t + \int_0^T \int_\Omega r^{k+1/2} \mathcal{O}(\tau^2) \eta \dd x \dd t \to \int_0^T \int_\Omega r(t) g_1(t) \eta \dd x \dd t.
    \]
    Finally, for any $\eta \in H^1([0,T];H^1(\Omega))$, by integration by parts we have  
    \[
        \int_0^T\lscal{\f{\p P}{\p t}}{\eta}\,\dd t = -\int_0^T \lscal{P}{\f{\p \eta }{\p t}}\,\dd t + \lscal{P(T)}{\eta(T)} - \lscal{P(0)}{\eta(0)}.  
    \] 
    Hence, from the regularity of $\eta$ and the convergence \eqref{eq:strong-p} , we obtain
    \[
        \int_0^T \lscal{P}{\f{\p \eta}{\p t}}\,\dd t  \to \int_0^T \scal{p}{\f{\p \eta}{\p t}}\,\dd t\quad \text{as}\quad \tau \to 0\quad\text{and }\quad \forall \eta  \in H^1([0,T];H^1(\Omega)).
    \]
    Gathering the previous convergences, we have
    \[
        \lscal{p(T)}{\eta(T)} - \lscal{p(0)}{\eta(0)} -  \int_0^T \scal{p}{\f{\p \eta}{\p t}}\,\dd t =  \int_0^T \int_\Omega \nabla q \nabla \eta + \left(V(x)q + \f{\p F(\abs{p}^2,\abs{q}^2)}{\p q}\right) \eta \dd x \dd t.
    \]
    Since $\nabla q + \left(V(x)q+ \f{\p F(\abs{p}^2,\abs{q}^2)}{\p q}\right)\in L^2(\Omega)$ which follow from the conservation of the Hamiltonian energy, we know that $p\in H^1\left([0,T]; H^{-1}(\Omega)  \right) $. 
    Finally, we find the first equation of the limit system~ \eqref{eq:limit-sys} and the same arguments can be applied to the second equation. This yields the result. 
    
    \hfill \qed
\end{proof}

\section{Error analysis}

In this section we analyse the difference between the exact and modified Hamiltonian, and establish a bound on
\[
    \abs{H[p(t^k),q(t^k)] - \tilde{H}[P^k,Q^k]}.  
\]
In addition we prove second-order convergence of the fully discrete SAV scheme \eqref{eq:SAV-schro-fulldiscr} approximating the solution of the nonlinear Schr\"odinger equation \eqref{eq:schro}. We introduce the following notation to study the error 
\begin{equation}
    e_u^k = \theta_u^k + \rho_u^k,
    \label{eq:decomp-error}
\end{equation}
where 
\[
    \theta_u^k = U^k - (I_N u)(t^k,x),\quad \rho_u^k = (I_N u)(t^k,x)- u(t^k,x).  
\]
For our convergence result we assume that the solution $u$  of \eqref{eq:schro} is sufficiently smooth satisfying
\begin{equation}
    \norm{\p_{ttt}u}_{L^\infty\left(0,T;H^1(\Omega)\right)} + \norm{u}_{L^\infty\left(0,T;H^2(\Omega)\right)} \le C.
    \label{eq:assumption-sol}
\end{equation} 
We define the different truncation errors by 
\begin{align*}
    T^{k+\f12}_u &= \f{u^\kk-u^k}{\Delta t}-\p_t u(t^{k+\f12}),\\
    \ov T^{k+\f12}_u &= u^{k+\f12}-u(t^{k+\f12}) = \f{u^\kk+u^k}{2} -u(t^{k+\f12}). \\
\end{align*}
We commence with two important lemma that will be useful in the global error analysis.
\begin{lemma}[Boundedness of nonlinear functions]
    If $(p,q)$ is a solution of \eqref{eq:SAV-schro} satisfying \eqref{eq:assumption-sol}, we have for $i=1,2$
    \[
        \abs{g_i(p,q)}, \abs{\f{\p g_i}{\p p}}, \abs{\f{\p g_i}{\p q}}, \abs{\f{\p^2 g_i}{\p p \p q}} ,  \abs{\f{\p^2 g_i}{\p p^2}}, \abs{\f{\p^2 g_i}{\p q^2}} \le C.
    \]
    \label{lem:bound-continuous-g}
\end{lemma}
\begin{proof}
\AP{
This result is found by a combination of the fact that $u\in L^\infty(0,T;H^2(\Omega))$, Remark~\ref{eq:linf-bound}, and assumption~\eqref{eq:assum-bound-g'}.}
\end{proof}
\begin{remark}
\label{rem:born-r}
From Lemma \ref{lem:bound-continuous-g}, and the hypothesis \eqref{eq:assumption-sol}, we know that
\[
    \abs{\p_{ttt} r} \le C \left(\norm{\p_{ttt} p}_0^2 + \norm{\p_{ttt} q}_0^2  \right).
\] 
\end{remark}
\AP{We have the following Lemma on the norm of the truncation errors (see Lemma 4.7 in~\cite{Wang} for example).}
\begin{lemma}[Truncation errors]
    \label{lem:error-trunc}
    For $\alpha = -1,0,1,2$, we have 
    \begin{equation*}
        \begin{aligned}
            \norm{T^{k+\f12}_\psi}_{H^\alpha(\Omega)}^2 \le \tau^3 \int_{t^k}^{t^\kk}\norm{\p_{ttt}\psi(s)}^2_{H^\alpha(\Omega)}\dd s ,\\
            \norm{\ov T^{k+\f12}_\psi}_{H^\alpha(\Omega)}^2 \le \tau^3 \int_{t^k}^{t^\kk} \norm{\p_{ttt}\psi(s)}^2_{H^\alpha(\Omega)}\dd s ,\\
        \end{aligned}
    \end{equation*}
\end{lemma}

\begin{theorem}[Error analysis]\label{thm:main}
    Assume that the solution of \eqref{eq:SAV-ham} satisfies \eqref{eq:assumption-sol} with initial condition 
    $
        u^0 \in H^3(\Omega).
    $
    Then the discrete solution $\{P^\kk,Q^\kk\}$ of the fully discrete SAV scheme  \eqref{eq:SAV-schro-fulldiscr} satisfies the error estimate
    \begin{equation*}
        \f12 \norm{\nabla e_q^{k+1}}^2_{0} + \f12 \norm{\nabla e_p^{k+1}}^2_{0} + \abs{e_r^{\kk}}^2 \le C \exp\left(\left[1-C\tau \right]^{-1} t^{\kk}\right)\left( \tau^4 + N^{-4}\right),  
    \end{equation*}
    where the constant $C$ depends on the smoothness of the solution \eqref{eq:assumption-sol}.

\end{theorem}
\begin{proof}

    \textit{Step 1: Error equations. }We begin by evaluating the model \eqref{eq:SAV-schro} at time $t^{k+1/2}$
    \begin{equation*}
        \begin{cases}
            \p_t p(t^{k+1/2}) &= -\Delta q(t^{k+1/2}) + r(t^{k+1/2}) g_1(t^{k+1/2}),\\
            \p_t q(t^{k+1/2}) &= \Delta p(t^{k+1/2}) - r(t^{k+1/2})g_2(t^{k+1/2}), \\
            \f{\dd r}{\dd t}(t^{k+1/2})&= \frac{1}{2} \left[\scal{g_1(t^{k+1/2})}{\p_t q(t^{k+1/2})} + \scal{g_2(t^{k+1/2})}{\p_t p(t^{k+1/2})} \right].
        \end{cases}
    \end{equation*}
    Subtracting the above equations from  \eqref{eq:SAV-schro-fulldiscr} yields 
    \begin{equation}
        \begin{cases}
            \f{e_p^{\kk}-e_p^k}{\tau} + T_p^{k+1/2} &= -\Delta\left(e_q^{k+1/2} +  \ov T^{k+1/2}_q \right)  +R^{k+1/2} \tilde{G}^{k+1/2}_1 - r(t^{k+1/2}) g_1(t^{k+1/2}),\\
            \f{e_q^{\kk}-e_q^k}{\tau} + T_q^{k+1/2} &= \Delta \left(e_p^{k+1/2}+\ov T^{k+1/2}_p \right)- R^{k+1/2}\tilde G_2^{k+1/2} + r(t^{k+1/2})g_2(t^{k+1/2}), \\
            \f{e_r^\kk-e_r^k}{\tau} + T^{k+1/2}_r&= \frac{1}{2} \big[ \scal{\tilde G_1^{k+1/2}}{\f{Q^{k+1}-Q^k}{\tau}} + \scal{\tilde G_2^{k+1/2}}{\f{P^{k+1}-P^k}{\tau}}\\ & -\scal{g_1(t^{k+1/2})}{\p_t q(t^{k+1/2})} - \scal{g_2(t^{k+1/2})}{\p_t p(t^{k+1/2})} \big].
        \end{cases}
        \label{eq:step-1-eqs}
    \end{equation}
    We introduce the error
    \[
        e_{g,1}^{k+1/2} = \tilde G_1^{k+1/2}-g_1(t^{k+1/2}).  
    \]
    The rightmost terms of the two first equations of \eqref{eq:step-1-eqs} can be replaced by 
    \begin{equation}
        R^{k+1/2}\tilde{G}_1^{k+1/2} - r(t^{k+1/2})g_1(t^{k+1/2}) = \tilde{G}_1^{k+1/2}\left(e^{k+1/2}_r  + \ov T^{k+1/2}_r \right) + r(t^{k+1/2})e_{g,1}^{k+1/2},  
        \label{eq:error-eq-tricks1}
    \end{equation}
    and 
    \begin{equation}
        R^{k+1/2}\tilde{G}_2^{k+1/2} - r(t^{k+1/2})g_2(t^{k+1/2}) = \tilde{G}_2^{k+1/2}\left(e^{k+1/2}_r  + \ov T^{k+1/2}_r \right) + r(t^{k+1/2})e_{g,2}^{k+1/2}. 
        \label{eq:error-eq-tricks2}
    \end{equation}
    Similarly, we have 
    \begin{equation}
        \begin{aligned}
            \f12&\left[\scal{\tilde G_1^{k+1/2}}{\f{Q^{k+1}-Q^k}{\tau}}   -\scal{g_1(t^{k+1/2})}{\p_t q(t^{k+1/2})} \right] \\
            &= \f12\left[\scal{\tilde{G}_1^{k+1/2}}{\f{e_q^{k+1}-e_q^k}{\tau}+T_q^{k+1/2}} + \scal{e_{g,1}^{k+1/2}}{\p_t q(t^{k+1/2})}  \right] ,
        \end{aligned}
        \label{eq:error-eq-tricks3}
    \end{equation}
    and
    \begin{equation}
        \begin{aligned}
            \f12 &\left[\scal{\tilde G_2^{k+1/2}}{\f{P^{k+1}-P^k}{\tau}} - \scal{g_2(t^{k+1/2})}{\p_t p(t^{k+1/2})}\right]\\
            &=\f12\left[  \scal{\tilde{G}_2^{k+1/2}}{\f{e_p^{k+1}-e_p^k}{\tau}+T_p^{k+1/2}} + \scal{e_{g,2}^{k+1/2}}{\p_t p(t^{k+1/2})}   \right].
        \end{aligned}
        \label{eq:error-eq-tricks4}
    \end{equation}
    Plugging \eqref{eq:error-eq-tricks1}, \eqref{eq:error-eq-tricks2},\eqref{eq:error-eq-tricks3}, and \eqref{eq:error-eq-tricks4} into \eqref{eq:step-1-eqs}, we thus obtain
    \begin{equation*}
        \begin{cases}
            \f{e_p^{\kk}-e_p^k}{\tau} + T_p^{k+1/2} &= -\Delta\left(e_q^{k+1/2} +  \ov T^{k+1/2}_q \right)  +\tilde{G}_1^{k+1/2}\left(e^{k+1/2}_r  + \ov T^{k+1/2}_r \right) + r(t^{k+1/2})e_{g,1}^{k+1/2},\\
            \f{e_q^{\kk}-e_q^k}{\tau} + T_q^{k+1/2} &= \Delta \left(e_p^{k+1/2}+\ov T^{k+1/2}_p \right) - \tilde{G}_2^{k+1/2}\left(e^{k+1/2}_r + \ov T^{k+1/2}_r \right) - r(t^{k+1/2})e_{g,2}^{k+1/2}, \\
            \f{e_r^\kk-e_r^k}{\tau} + T^{k+1/2}_r&=\f12\big[\scal{\tilde{G}_1^{k+1/2}}{\f{e_q^{k+1}-e_q^k}{\tau}+T_q^{k+1/2}} + \scal{e_{g,1}^{k+1/2}}{\p_t q(t^{k+1/2})}\\
            & +\scal{\tilde{G}_2^{k+1/2}}{\f{e_p^{k+1}-e_p^k}{\tau}+T_p^{k+1/2}} + \scal{e_{g,2}^{k+1/2}}{\p_t p(t^{k+1/2})}  \big].
        \end{cases}
    \end{equation*}
    Using the decomposition of the error \eqref{eq:decomp-error}, we furthermore obtain 
    \begin{equation}
        \begin{cases}
            \f{\theta_p^{\kk}-\theta_p^k}{\tau} + \Delta\left(\f{\theta_q^\kk+\theta_q^k}{2}\right)  &= -\f{\rho_p^{\kk}-\rho_p^k}{\tau}  -\Delta\left(\rho_q^{k+1/2} +  \ov T^{k+1/2}_q \right)  +\tilde{G}_1^{k+1/2}\left( e^{k+1/2}_r  + \ov T^{k+1/2}_r \right) \\
            &+ r(t^{k+1/2})e_{g,1}^{k+1/2}- T_p^{k+1/2},\\
            \f{\theta_q^{\kk}-\theta_q^k}{\tau}-\Delta \left(\f{\theta_p^\kk+\theta_p^k}{2}\right)  &= -\f{\rho_q^{\kk}-\rho_q^k}{\tau}+ \Delta \left(\rho_p^{k+1/2}+\ov T^{k+1/2}_p \right)- \tilde{G}_2^{k+1/2}\left(e^{k+1/2}_r + \ov T^{k+1/2}_r \right) \\
            &- r(t^{k+1/2})e_{g,2}^{k+1/2}- T_q^{k+1/2}, \\
            \f{e_r^\kk-e_r^k}{\tau} + T^{k+1/2}_r&=\f12\big[\scal{\tilde{G}_1^{k+1/2}}{\f{e_q^{k+1}-e_q^k}{\tau}+T_q^{k+1/2}} + \scal{e_{g,1}^{k+1/2}}{\p_t q(t^{k+1/2})}\\
            & +\scal{\tilde{G}_2^{k+1/2}}{\f{e_p^{k+1}-e_p^k}{\tau}+T_p^{k+1/2}} + \scal{e_{g,2}^{k+1/2}}{\p_t p(t^{k+1/2})}  \big].
        \end{cases}
        \label{eq:step-1-eqs3}
    \end{equation}

    \textit{Step 2. Error estimate formula. }
    We use the following notations to make the results more compact
    \[
        D_\tau^1 \theta_p^{\kk} =  \f{\theta_p^{\kk}-\theta_p^k}{\tau},\quad
        D^1 \theta_p^{\kk} =  \theta_p^{\kk}-\theta_p^k.
    \]

    Taking the inner product of the first equation of the system \eqref{eq:step-1-eqs3} with $-D^1 \theta_q^\kk$ and the second with $ D^1 \theta_p^\kk$, and summing the results, we also have, 
    \begin{equation}
        \begin{aligned}
        \f12D^1\norm{\nabla \theta_q^{k+1}}^2_{0} + \f12D^1 \norm{\nabla\theta_p^{k+1}}^2_{0} &= \scal{D_\tau^1 \rho_p^{\kk}}{D^1\theta_q^\kk} - \scal{D_\tau^1 \rho_q^{\kk}}{D^1 \theta_p^\kk} \\
        &- \scal{\nabla\rho_q^{k+1/2}}{\nabla D^1\theta_q^\kk}  - \scal{\nabla \rho_p^{k+1/2}}{\nabla D^1 \theta_p^\kk} \\
        &- \scal{\nabla \ov T^{k+1/2}_q }{\nabla D^1\theta_q^\kk} - \scal{\nabla \ov T^{K+1/2}_p}{\nabla D^1 \theta_p^\kk } \\
        & - \scal{\tilde{G}_1^{k+1/2}\left(e^{k+1/2}_r + \ov T^{k+1/2}_r \right)}{D^1\theta_q^\kk} \\
        &- \scal{\tilde{G}_2^{k+1/2}\left(e^{k+1/2}_r 
        + \ov T^{k+1/2}_r \right)}{D^1 \theta_p^\kk}  
        \\
        &-\scal{r(t^{k+1/2})e_{g,1}^{k+1/2}- T_p^{k+1/2}}{D^1\theta_q^\kk} \\
        &- \scal{r(t^{k+1/2})e_{g,2}^{k+1/2}+ T_p^{k+1/2}}{D^1 \theta_p^\kk}.
        \end{aligned}
        \label{eq:error-formula3}
    \end{equation}
    Multiplying the third equation of \eqref{eq:step-1-eqs3} by $2 \tau e^{k+1/2}_r$, we have 
    \begin{equation}
        \begin{aligned}
        D^1\abs{e_r^\kk}^2  &+ 2\tau T^{k+1/2}_r e^{k+1/2}_r - \tau {e^{k+1/2}_r} \big[\scal{\tilde{G}_1^{k+1/2}}{D^1_\tau \rho_q^\kk +T^{k+1/2}_q} \\
        &+ \scal{\tilde{G}_2^{k+1/2}}{D^1_\tau \rho_p^\kk+T^{k+1/2}_p} +  \scal{e^{k+1/2}_{g,1}}{\p_t q(t^{k+1/2})} + \scal{e^{k+1/2}_{g,2}}{\p_t p(t^{k+1/2})} \big] \\
        &=   {e^{k+1/2}_r} \left[\scal{\tilde{G}_1^{k+1/2}}{D^1 \theta_q^\kk} + \scal{\tilde{G}_2^{k+1/2}}{D^1 \theta_p^\kk} \right].
        \end{aligned}
        \label{eq:error-formula4}
    \end{equation}
    Using \eqref{eq:error-formula4} in \eqref{eq:error-formula3}, we have 
    \begin{equation}
        \begin{aligned}
        \f12 D^1&\norm{\nabla \theta_q^{k+1}}^2_{0} + \f12D^1 \norm{\nabla\theta_p^{k+1}}^2_{0} + D^1\abs{e_r^{\kk}}^2  \\
        &= \scal{D_\tau^1 \rho_p^{\kk}}{D^1\theta_q^\kk} - \scal{D_\tau^1 \rho_q^{\kk}}{D^1 \theta_p^\kk} \\
        &- \scal{\nabla\left(\rho_q^{k+1/2}+ \ov T^{k+1/2}_q\right)}{ \nabla D^1\theta_q^\kk}  - \scal{\nabla \left(\rho_p^{k+1/2}+\ov T^{k+1/2}_p\right)}{\nabla D^1\theta_p^\kk} \\
        &-\ov T_r^{k+1/2}\left[\scal{\tilde G_1^{k+1/2}}{D^1\theta_q^\kk}+ \scal{\tilde G_2^{k+1/2}}{D^1\theta_p^\kk}\right]\\
        & -  2 \tau T^{k+1/2}_r e^{k+1/2}_r + \tau{e^{k+1/2}_r} \big[\scal{\tilde{G}_1^{k+1/2}}{D^1_\tau \rho_q^\kk+T^{k+1/2}_q} \\
        &+ \scal{\tilde{G}_2^{k+1/2}}{D^1_\tau\rho_p^\kk+T^{k+1/2}_p} +  \scal{e^{k+1/2}_{g,1}}{\p_t q(t^{k+1/2})} + \scal{e^{k+1/2}_{g,2}}{\p_t p(t^{k+1/2})} \big]  
        \\
        &-\scal{r(t^{k+1/2})e_{g,1}^{k+1/2}- T_p^{k+1/2}}{D^1\theta_q^\kk} \\
        &- \scal{r(t^{k+1/2})e_{g,2}^{k+1/2}+ T_p^{k+1/2}}{D^1\theta_p^\kk}.
        \end{aligned}
        \label{eq:error-formula5}
    \end{equation}

    \textit{Step 3. Inequalities for the terms on the right-hand side of \eqref{eq:error-formula5}. } 

    Now, we bound the right-hand side of \eqref{eq:error-formula5}. Using    Lemma \ref{lem:interp-error}, Lemma \ref{lem:error-trunc} and Young's inequality we have
    \begin{equation*}
        \begin{aligned}
            \scal{D^1_\tau \rho_p^{\kk}}{ D^1\theta_q^{k+1}} &\le 4 \norm{D^1_\tau \rho_p^{\kk}}^2_0 + \f1{16} \norm{D^1\theta_q^{k+1}}_0^2 \le C N^{-6} + \f1{16} \norm{D^1\theta_q^{k+1}}_0^2,
        \end{aligned}
    \end{equation*}
    \begin{equation*}
        \begin{aligned}
            -\scal{D_\tau^1\rho_q^{\kk}}{D^1\theta_p^{k+1}} \le 4 \norm{D_\tau^1\rho_q^{\kk}}_0^2 + \f1{16}\norm{D^1\theta_p^{k+1}}_0^2 \le C N^{-6} + \f1{16}\norm{D^1 \theta_p^{k+1}}_0^2.
        \end{aligned}
    \end{equation*}
    Then, from Theorem \eqref{th:preservation-ham}, we have 
    \begin{equation*}
        \begin{aligned}
            - &\scal{\nabla\rho_q^{k+1/2}}{\nabla D^1\theta_q^{k+1}}  - \scal{\nabla \rho_p^{k+1/2}}{\nabla D^1 \theta_p^{k+1}} \\
            &\le  \left(\norm{\nabla \rho_q^{k+1/2}}_{0}^2 \norm{\nabla D^1  \theta_q^{k+1}}^2_0 +\norm{\nabla \rho_p^{k+1/2}}_{0}^2 \norm{ \nabla D^1 \theta_p^{k+1}}^2_0 \right) \\
            &\le C N^{-4} ,
        \end{aligned}
    \end{equation*}
    and
    \begin{equation*}
        \begin{aligned}
            - &\scal{\nabla \ov T^{k+1/2}_q }{ \nabla D^1\theta_q^{k+1}} - \scal{\nabla \ov T^{K+1/2}_p}{ \nabla D^1\theta_p^{k+1}} \\
            &\le \left(\norm{\nabla \ov T^{k+1/2}_q }_0^2  \norm{\nabla D^1 \theta_q^{k+1}}^2_0+ \norm{\nabla \ov T^{K+1/2}_p}_0^2\norm{\nabla D^1\theta_p^{k+1}}^2_0 \right)  \\
            &\le C\tau^4 ,
        \end{aligned}
    \end{equation*}
    For the rest of the terms on the right-hand side of \eqref{eq:error-formula5}, we use Lemma \ref{lem:error-trunc}, and Proposition \ref{prop:H2-bound} together with Lemma \ref{lem:bound-continuous-g}, and Remark \ref{rem:born-r}, to obtain 
    \begin{equation}
        \begin{aligned}
            -&\ov T_r^{k+1/2}\left[\scal{\tilde G_1^{k+1/2}}{D^1\theta_q^\kk}+ \scal{\tilde G_2^{k+1/2}}{D^1\theta_p^\kk}\right]\\
            & \le 4\abs{\ov T_r^{k+1/2}}^2\left(\norm{\tilde G_1^{k+1/2}}_0^2+\norm{\tilde G_1^{k+1/2}}_0^2\right) + \f1{16}\left(\norm{D^1\theta_q^\kk}_0^2 + \norm{D^1 \theta_p^\kk}_0^2 \right)\\
            &\le C \tau^4 + \f1{16}\left(\norm{D^1\theta_q^\kk}_0^2 + \norm{D^1\theta_p^\kk}_0^2 \right),
        \end{aligned}   
        \label{eq:ineq-err-r1}
    \end{equation}
    \begin{equation*}
        \begin{aligned}
            -  2\tau T^{k+1/2}_r e^{k+1/2}_r  \le C\tau \left( \norm{T^{k+1/2}_r}_0^2 + \abs{e^{k+1}_r}^2 + \abs{e^k_r}^2\right)\le C \tau^5 + \tau \abs{e^{k+1}_r}^2 + \tau\abs{e^k_r}^2,
        \end{aligned}
    \end{equation*}
    \begin{equation*}
        \begin{aligned}
            \tau{e^{k+1/2}_r}& \scal{\tilde{G}_1^{k+1/2}}{D^1_\tau\rho_q^\kk + T^{k+1/2}_q}\\
            &\le  \f\tau2 \norm{\tilde{G}_1^{k+1/2}}_0^2 \left( \norm{D^1_\tau\rho_q^\kk}_0^2 + \norm{ T^{k+1/2}_q}_0^2 + \abs{e_r^\kk}^2 + \abs{e_r^k}^2\right) \\
            &\le C \tau \left(N^{-6} + \tau^4 + \abs{e_r^\kk}^2 + \abs{e_r^k}^2 \right),
        \end{aligned}
    \end{equation*}
    \begin{equation*}
        \begin{aligned}
            \tau{e^{k+1/2}_r} &\scal{\tilde{G}_2^{k+1/2}}{D^1_\tau\rho_p^\kk+ T^{k+1/2}_p}\\
            &\le  \f\tau2 \norm{\tilde{G}_2^{k+1/2}}_0^2 \left( \norm{D^1_\tau\rho_p^\kk}_0^2 + \norm{ T^{k+1/2}_p}^2_0+ \abs{e_r^\kk}^2 + \abs{e_r^k}^2\right) \\
            &\le C\tau \left(N^{-6} + \tau^4+ \abs{e_r^\kk}^2 + \abs{e_r^k}^2 \right),
        \end{aligned}
    \end{equation*}

    \begin{equation*}
        \begin{aligned}
            \tau {e^{k+1/2}_r}& \left[\scal{e^{k+1/2}_{g,1}}{\p_t q(t^{k+1/2})} + \scal{e^{k+1/2}_{g,2}}{\p_t p(t^{k+1/2})} \right]\\
            &\le \f\tau2\norm{\p_t q(t^{k+1/2})}^2_0\left( \norm{e^{k+1/2}_{g,1}}^2_0+ \abs{e_r^\kk}^2 + \abs{e_r^k}^2 \right) \\
            &+  \f\tau2\norm{\p_t p(t^{k+1/2})}^2_0\left( \norm{e^{k+1/2}_{g,2}}^2_0+ \abs{e_r^\kk}^2 + \abs{e_r^k}^2 \right) \\
            &\le  C\tau \left(\norm{e^{k+1/2}_{g,1}}^2_0+ \norm{e^{k+1/2}_{g,2}}^2_0 + \abs{e_r^\kk}^2 + \abs{e_r^k}^2  \right),
        \end{aligned}
    \end{equation*}
    and
    \begin{equation}
        \begin{aligned}
            -&\scal{r(t^{k+1/2})e_{g,1}^{k+1/2}- T_p^{k+1/2}}{D^1\theta_q^{k+1}} - \scal{r(t^{k+1/2})e_{g,2}^{k+1/2}+ T_p^{k+1/2}}{D^1\theta_p^{k+1}}\\
            &\le 4\abs{r(t^{k+1/2}}^2 \left(\norm{e_{g,1}^{k+1/2}}^2_0 +\norm{e_{g,2}^{k+1/2}}^2_0 + \norm{T_p^{k+1/2}}^2_0+ \norm{T_q^{k+1/2}}^2_0  \right) \\
            &+ \f1{16}\left( \norm{D^1\theta_q^{k+1}}^2_0 +\norm{D^1\theta_p^{k+1}}^2_0 \right)\\
            &\le C \left(\norm{e_{g,1}^{k+1/2}}^2_0 +\norm{e_{g,2}^{k+1/2}}^2_0 + 2\tau^4\right)+  \f1{16}\left( \norm{D^1\theta_q^{k+1}}^2_0 +\norm{D^1\theta_p^{k+1}}^2_0 \right).
        \end{aligned}
        \label{eq:ineq-err-r6}
    \end{equation}
    \textit{Step 4. Estimating the terms in the inequalities ~\eqref{eq:ineq-err-r1}--\eqref{eq:ineq-err-r6}.} First, we aim to emiminate the terms $\norm{D^1\theta_p^{k+1}}^2_0$ and $\norm{D^1\theta_q^{k+1}}^2_0$ in the above inequalities. Taking the inner product of the first equation of \eqref{eq:step-1-eqs3} with $2\tau \theta_p^\kk$, we obtain  
    \begin{equation*}
        \begin{aligned}
        \scal{D^1_\tau\theta_p^{\kk}}{2\tau \theta_p^\kk}  &= 2\tau \scal{\nabla \theta_q^{k+1/2}}{\nabla  \theta_p^\kk}-2\tau\scal{D^1_\tau\rho_p^{\kk}}{ \theta_p^\kk}  +2\tau\scal{\nabla\left(\rho_q^{k+1/2} +  \ov T^{k+1/2}_q \right)}{ \nabla \theta_p^\kk}  \\
        &+ 2\tau \left(e^{k+1/2}_r+\ov T^{k+1/2}_r\right)\scal{\tilde{G}_1^{k+1/2}}{\theta_p^\kk}  
        + 2\tau \scal{r(t^{k+1/2})e_{g,1}^{k+1/2}- T_p^{k+1/2}}{\theta_p^\kk}.
        \end{aligned}
    \end{equation*}
    Knowing that 
    \[
        \scal{D^1_\tau\theta_p^{\kk}}{2\tau \theta_p^\kk} \ge \norm{D^1\theta^\kk_p}^2_0,
    \]
    we have 
    \begin{equation}
        \begin{aligned}
            \norm{D^1\theta^\kk_p}^2_0  &\le 2\tau \scal{\nabla \theta_q^{k+1/2}}{\nabla  \theta_p^\kk}-2\tau\scal{D^1_\tau\rho_p^{\kk}}{ \theta_p^\kk}  +2\tau\scal{\nabla\left(\rho_q^{k+1/2} +  \ov T^{k+1/2}_q \right)}{ \nabla \theta_p^\kk}  \\
        &+ 2\tau \left(e^{k+1/2}_r+\ov T^{k+1/2}_r\right)\scal{\tilde{G}_1^{k+1/2}}{\theta_p^\kk} + 2\tau \scal{r(t^{k+1/2})e_{g,1}^{k+1/2}- T_p^{k+1/2}}{\theta_p^\kk}.
        \end{aligned}
        \label{eq:estimate-Dtetha}
    \end{equation}
    Let us bound the terms on the right-hand side of \eqref{eq:estimate-Dtetha}. Using  Lemma \ref{lem:interp-error} we find that
    \begin{align*}
        2\tau \scal{\nabla \theta_q^{k+1/2}}{\nabla  \theta_p^\kk} \le \tau \left( \norm{\nabla \theta_q^{k+1}}_0^2 + \norm{\nabla \theta_q^{k}}_0^2+ \norm{\nabla \theta_p^\kk}_0^2\right),\\
        -2\tau\scal{D^1_\tau\rho_p^{\kk}}{ \theta_p^\kk} \le \tau\left( \norm{D^1_\tau\rho_p^{\kk}}_{H^{-1}(\Omega)}^2 + \norm{\nabla \theta_p^\kk}^2_0\right)\le \tau\left( N^{-6} + \norm{\nabla \theta_p^\kk}^2_0\right), \\
        2\tau\scal{\nabla\left(\rho_q^{k+1/2} +  \ov T^{k+1/2}_q \right)}{ \nabla \theta_p^\kk} \le \tau \left( C N^{-4}+C\tau^4 + \norm{\nabla \theta_p^\kk}_0^2\right),\\
        2\tau \scal{r(t^{k+1/2})e_{g,1}^{k+1/2}- T_p^{k+1/2}}{\theta_p^\kk}\le \tau \left(C\norm{e_{g,1}^{k+1/2}}_0^2 + \tau^4 + C\norm{\nabla \theta_p^{k+1}}^2_0 \right),
    \end{align*}
    where we have used the Poincaré inequality to obtain the last inequality. 
    Plugging the previous inequalities into \eqref{eq:estimate-Dtetha}, we obtain
    \begin{equation}
        \begin{aligned}
        \norm{D^1\theta^\kk_p}^2_0  &\le \tau \left( \norm{\nabla \theta_q^{k+1}}_0^2 + \norm{\nabla \theta_q^{k}}_0^2 + C N^{-4}+ C\tau^4 +C\norm{e_{g,1}^{k+1/2}}_0^2  + C\norm{\nabla \theta_p^{k+1}}^2_0 \right).\label{est1}
        \end{aligned}
    \end{equation} 

    Similarly, taking the inner product of the second equation of \eqref{eq:step-1-eqs3} with $2\tau \theta_q^\kk$ and repeating the same steps as before, we obtain 
    \begin{equation}
        \begin{aligned}
        \norm{D^1\theta^\kk_q}^2_0  &\le \tau \left( \norm{\nabla \theta_p^{k+1}}_0^2 + \norm{\nabla \theta_p^{k}}_0^2 + C N^{-4}+ C\tau^4 +C\norm{e_{g,2}^{k+1/2}}_0^2  + C\norm{\nabla \theta_q^{k+1}}^2_0 \right).
        \end{aligned}\label{est2}
    \end{equation} 

    \textit{Step 5. Estimating $\norm{e_{g,1}^{k+1/2}}_0^2$ and $\norm{e_{g,2}^{k+1/2}}_0^2$.} Using the notations
    \[
        S(p,q) = \sqrt{\mathcal{E}_1(p,q) + C},
    \]
    and
    \[
        N_1(p,q) = \f{\delta }{\delta q} \mathcal{E}_1(p,q), \quad N_2(p,q) = \f{\delta }{\delta p} \mathcal{E}_1(p,q)
    \]

    we have that 
    \begin{equation*}
        \begin{aligned}
            e_{g,1}^{k+1/2} &= G_1(P^{k+1/2},Q^{k+1/2}) - g_1(p(t^{k+1/2}),q(t^{k+1/2}))\\
            &= \f{N_1(P^{k+1/2},Q^{k+1/2})}{S(P^{k+1/2},Q^{k+1/2})} - \f{N_1(p(t^{k+1/2}),q(t^{k+1/2}))}{S(p(t^{k+1/2}),q(t^{k+1/2}))}\\
            &= \f{N_1(P^{k+1/2},Q^{k+1/2})}{S(P^{k+1/2},Q^{k+1/2})} - \f{N_1(P^{k+1/2},Q^{k+1/2})}{S(p(t^{k+1/2}),q(t^{k+1/2}))} + \f{N_1(P^{k+1/2},Q^{k+1/2})}{S(p(t^{k+1/2}),q(t^{k+1/2}))} \\
            &- \f{N_1(p(t^{k+1/2}),q(t^{k+1/2}))}{S(p(t^{k+1/2}),q(t^{k+1/2}))}\\
            &= \f{N_1(P^{k+1/2},Q^{k+1/2})\left[\mathcal{E}_1(p(t^{k+1/2}),q(t^{k+1/2})) - \mathcal{E}_1(P^{k+1/2},Q^{k+1/2}) \right]}{S(P^{k+1/2},Q^{k+1/2})S(p(t^{k+1/2}),q(t^{k+1/2}))\left[S(P^{k+1/2},Q^{k+1/2})+S(p(t^{k+1/2}),q(t^{k+1/2}))\right]} \\
            &+ \f{N_1(P^{k+1/2},Q^{k+1/2})-N_1(p(t^{k+1/2}),q(t^{k+1/2}))}{S(p(t^{k+1/2}),q(t^{k+1/2}))}.
        \end{aligned}
    \end{equation*}
    \AP{From the smoothness assumption \eqref{eq:assumption-sol}, Lemma \ref{lem:bound-continuous-g}, and Remark~\ref{rem:linf}, we have }
    \begin{equation*}
        \norm{e_{g,1}^{k+1/2} }^2_0 \le C\left[\norm{P^{k+1/2}-p(t^{k+1/2})}^2_0+\norm{Q^{k+1/2}-q(t^{k+1/2})}^2_0\right].
    \end{equation*}
    Then, using the notation \eqref{eq:decomp-error} and Lemma \ref{lem:error-trunc}, we obtain 
    \begin{equation*}
        \begin{aligned}
        \norm{e_{g,1}^{k+1/2} }^2_0 &\le C \left[ \norm{\theta^{k+1}_p}_0^2 + \norm{\theta^{k}_p}_0^2 + \norm{\theta^{k+1}_q}_0^2 + \norm{\theta^{k}_q}_0^2 + \tau^3 + N^{-4}\right]\\
        &\le     C \left[ \norm{\nabla \theta^{k+1}_p}_0^2 + \norm{\nabla \theta^{k}_p}_0^2 + \norm{\nabla \theta^{k+1}_q}_0^2 + \norm{\nabla \theta^{k}_q}_0^2 + \tau^3 + N^{-4}\right].
        \end{aligned}
    \end{equation*}
    Similarly, we have  
    \begin{equation*}
        \norm{e_{g,2}^{k+1/2} }^2_0 \le  C \left[ \norm{\nabla \theta^{k+1}_q}_0^2 + \norm{\nabla \theta^{k}_q}_0^2 + \norm{\nabla \theta^{k+1}_p}_0^2 + \norm{\nabla \theta^{k}_p}_0^2 + \tau^3 + N^{-4}\right].
    \end{equation*}
    \textit{Step 6. Discrete Gronwall Lemma.} The above  two estimates together with \eqref{est1} and \eqref{est2} imply
    \begin{equation*}
        \begin{aligned}
        \f12D^1&\norm{\nabla \theta_q^{k+1}}^2_{0} + \f12D^1 \norm{\nabla\theta_p^{k+1}}^2_{0} + D^1\abs{e_r^{\kk}}^2  \\
        &\le \tau C\big[\norm{\nabla \theta^{k+1}_p}_0^2 + \norm{\nabla \theta^{k}_p}_0^2 + \norm{\nabla \theta^{k+1}_q}_0^2 + \norm{\nabla \theta^{k}_q}_0^2 + \abs{e_r^{k+1}}^2 + \abs{e_r^{k}}^2\big] + C\left[\tau^4+N^{-4}\right].
        \end{aligned}
    \end{equation*}
    Therefore, by the use of Gronwall's Lemma, we can conclude that
    \begin{equation*}
        \f12\norm{\nabla \theta_q^{k+1}}^2_{0} + \f12\norm{\nabla\theta_p^{k+1}}^2_{0} + \abs{e_r^{\kk}}^2 \le C \exp\left(\left[1-C\tau \right]^{-1} t^{\kk}\right)\left( \tau^4 + N^{-4}\right).
    \end{equation*}
    \hfill\qed
\end{proof}

\section{Numerical experiments}\label{sec:num}
In this section we numerically confirm our theoretical convergence result given in Theorem \ref{thm:main} and illustrate the long time  energy conservation of the SAV method. 
\AP{In the following, the numerical results have been obtained from Algorithm 2. However, we want to emphasize that Algorithm 1 leads to the same results and has a comparable computational cost for $d=1$.}

Our numerical findings suggest the favorable energy preservation of the SAV method compared to classical splitting methods in certain applications such as for non-linearities with non-integer exponents which arise for instance in context of optical dark and power law solitons with surface plasmonic interactions \cite{crutcher_derivation_2011}. For the comparison we use the  classical first order Lie and second order Strang splitting which are known for their near energy preservation over long times,  see, e.g., \cite{Faou-geometric-2012}. 

In the numerical examples we plot the deviation of the exact Hamiltonian and the modified Hamiltonian  $\tilde H$, i.e., ${e}_{\tilde H}= \abs{H(t^k)-\tilde H^k}$, the error between the exact Hamiltonian and the discrete non-modified Hamiltonian
$e_H = \abs{H(u(t^k))- H(U^k)}$, as well as the $L^2$  error ${e}_u =\norm{\abs{U^k}-\abs{u(t^k)}}_{L^2(\Omega)}$. We choose the potential $V=0$ in the Schr\"odinger equation \eqref{eq:schro}.


\subsection{First test case: cubic nonlinearity}
In a first example we consider the nonlinear Schr\"odinger equation \eqref{eq:schro} with a cubic nonlinearity i.e.
\[
    f(\abs{u}^2) = \beta \abs{u}^2
\]
on the spatial domain $\Omega = [-32,32] $. In Figure  \ref{fig:order-cubic}  we choose a mesh size $h = 1/32$ and approximate the soliton solution \cite{Bao-Cubic-2013, antoine-SAV-2020}
\[
    u(x,t) = \f{a}{\sqrt{-\beta}}\text{sech}\left(a(x-vt) \right)\exp\left(ivx-0.5(v^2-a^2)t \right),
\] 
with the parameters $a=1$, $\beta=-1$ and $v=1$ up to $T=10$.  Figure \ref{fig:order-cubic} numerically confirms the second-order convergence of the SAV method. The numerical findings also suggest that the error constant of the Strang splitting method is slightly better than the one of the SAV method in this example.
\begin{figure}
  \begin{minipage}{.5\textwidth}
    \centering
    \includegraphics[width=.99\linewidth]{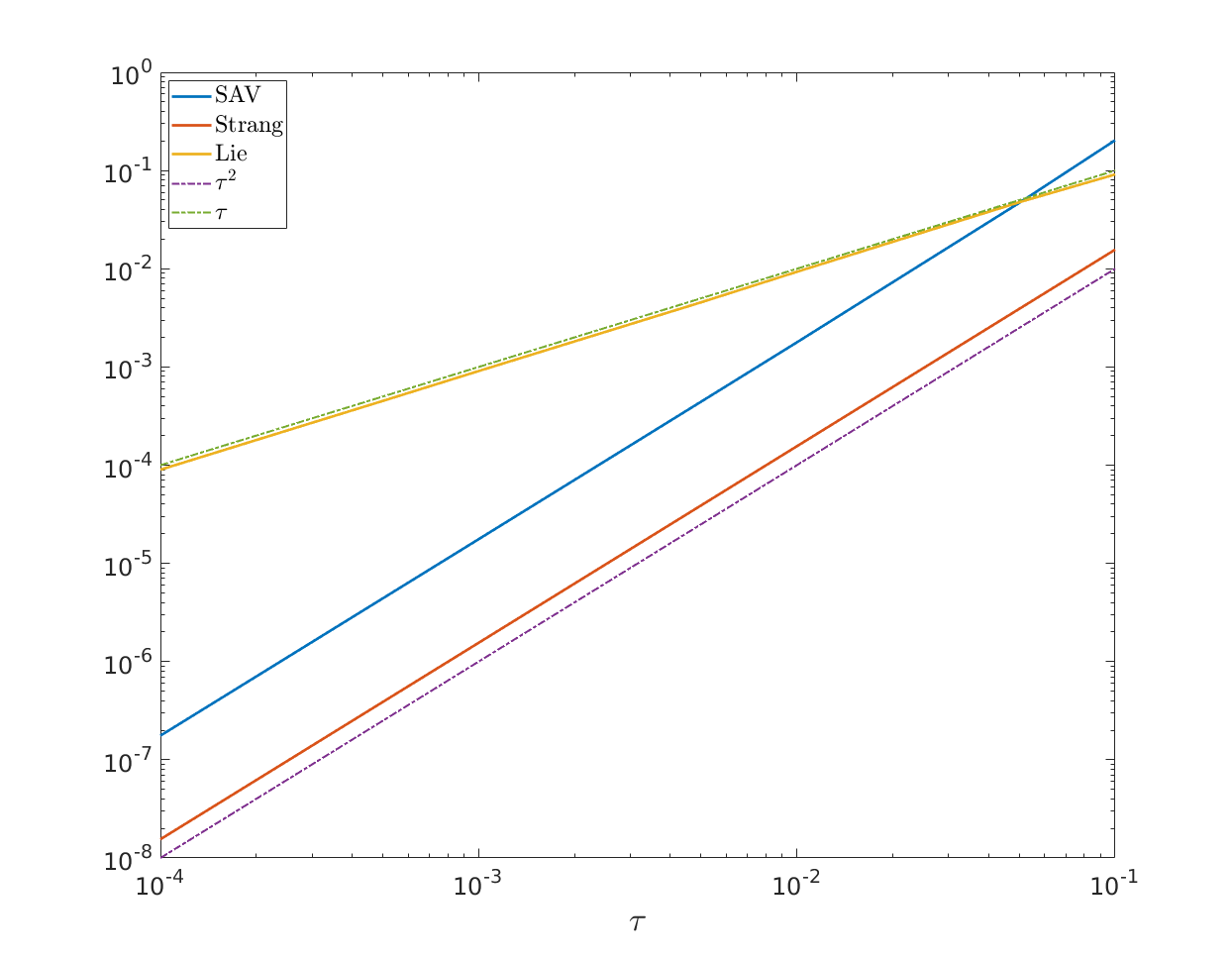}
  \end{minipage}%
  \begin{minipage}{.5\textwidth}
      \centering
      \includegraphics[width=.99\linewidth]{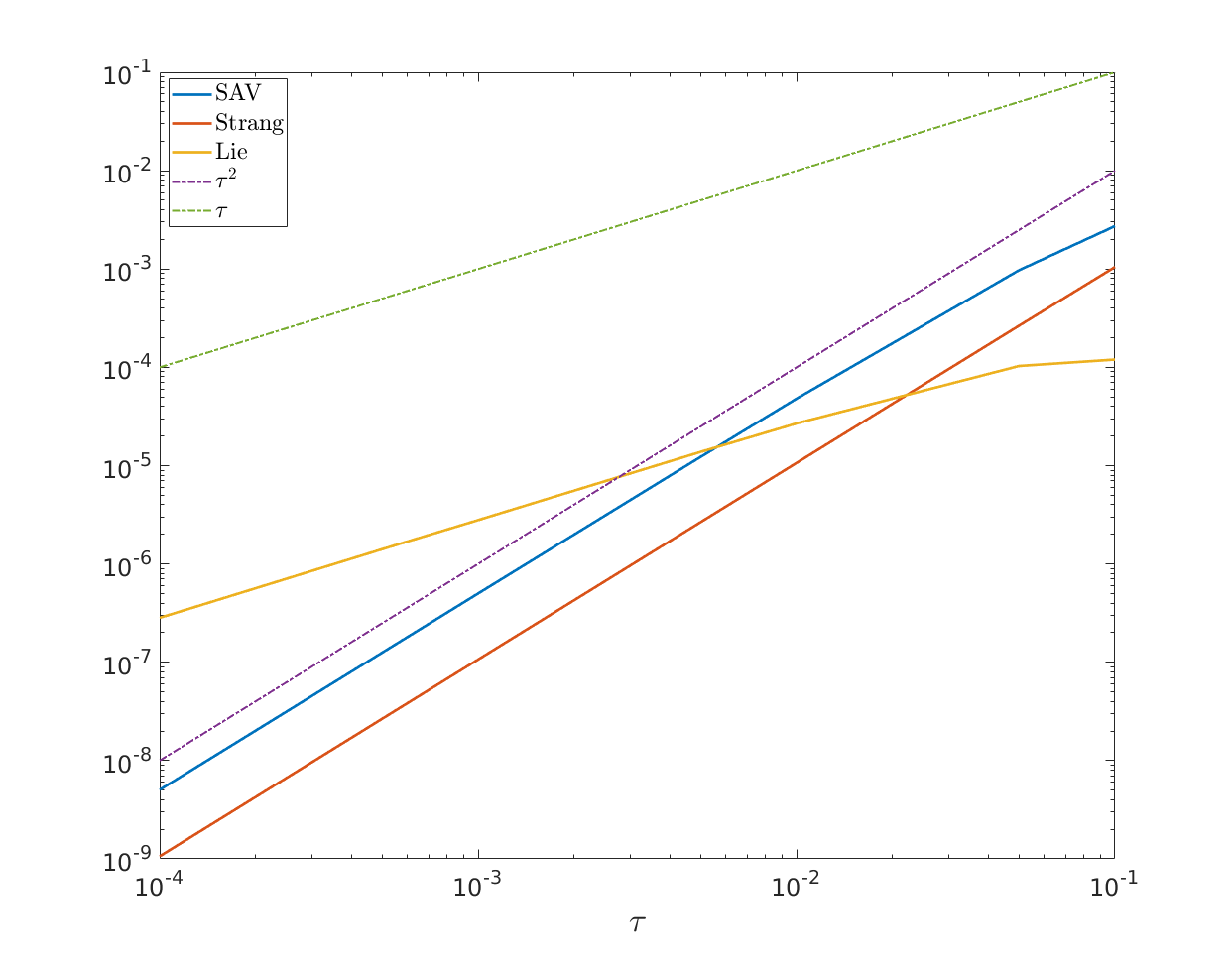}
  \end{minipage}
  \caption{Error ${e}_u$ (left) and $e_H$ (right) versus step size  $\tau$ at time $T=10$.}
  \label{fig:order-cubic}
\end{figure}
In Figure \ref{fig:cubic-evolv} we simulate the solitary wave 
\begin{equation*}
    u(t,x) = \f{\sqrt{2}e^{i t}}{\cosh(x)}
\end{equation*} 
on the domain $[-\f\pi{0.11}, \f \pi{0.11}]$ with $N=256$ collocation points and time step size $\tau = 0.01$. We illustrate the evolution of the errors $e_H$, $e_u$ and $e_{\tilde{H}}$ over long times, i.e., up to $T=1000$. Our numerical findings  confirm the conservation of the modified Hamiltonian by the SAV method, see Figure \ref{fig:cubic-evolv}.  We also observe that the SAV method preserves well the exact energy and $L^2$ norm over long times. Even though, the error $e_H$ of the Strang splitting seems favorable in this example, we have to stress that the modified Hamiltonian is closer to the value of the real Hamiltonian (see error $e_{\tilde H}$ on Figure \ref{fig:cubic-evolv}). 
\begin{figure}[h!] 
  \begin{minipage}{.5\textwidth}
    \centering
    \includegraphics[width=.99\linewidth]{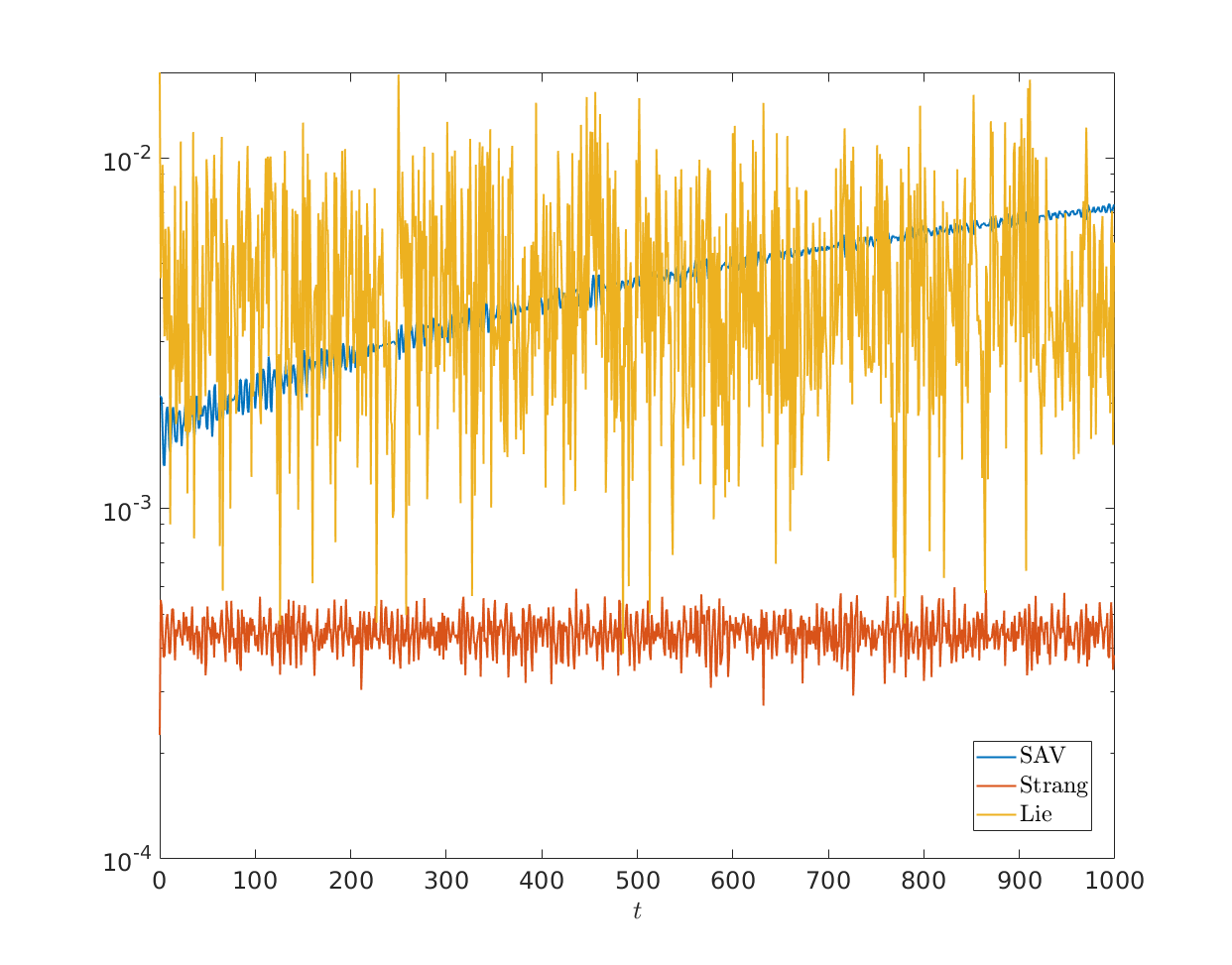}
  \end{minipage}%
  \begin{minipage}{.5\textwidth}
      \centering
      \includegraphics[width=.99\linewidth]{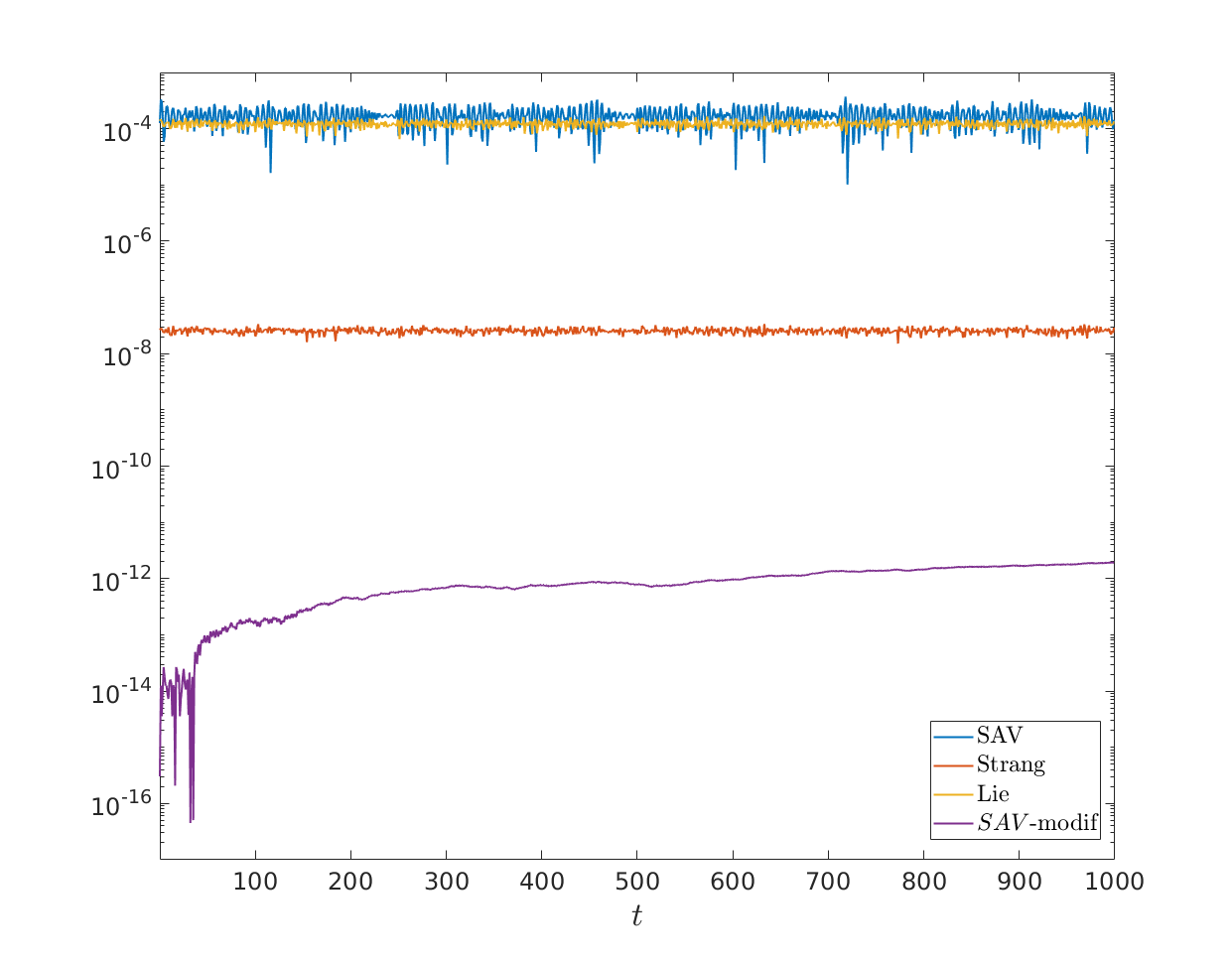}
  \end{minipage}
  \caption{Left Figure: error ${e}_u$ (left) through time. Right Figure: $e_H$ (blue, red, yellow) and $e_{\tilde H}$ (purple) through time.}
  \label{fig:cubic-evolv}
\end{figure}




\subsection{Second test case: cubic nonlinearity with non-smooth initial condition}
In this example we analyse the error behaviour of the SAV scheme in case of non-smooth initial data. For this purpose we solve the NLS equation with cubic nonlinearity with initial data of various regularity. More precisely, we choose $f(\abs{u}^2) = \beta \abs{u}^2$ with $\beta = 1$ and consider $u^0 \in H^\alpha$ with $\alpha = 3/2,2,3,5$ on the spatial  domain $\Omega = [-\pi, \pi]$ with  $N = 1024$ gridpoints. The discrete initial data of various regularity is generated  as proposed in \cite{schratz-low-2018}.


Figure \ref{fig:order-alpha} shows the convergence behaviour of the SAV scheme, and the two splitting methods for the initial data of different regularity. We find that if $\alpha < 3$, the SAV method does not maintain its second order convergence rate and for $\alpha = 2$, the SAV scheme reduces to first order. Decreasing the regularity of the initial condition even more, the convergence worsens and becomes less than  order $1$. A similar order reduction is observed for the splitting schemes, however, for the latter  the error starts to oscillate  for $\alpha < 3$.   Again, the error $e_H$ is favorable for the Strang splitting for all $\alpha$. However, the modified Hamiltonian is closer to the real Hamiltonian (see Figure \ref{fig:error-time-23} for $\alpha = \f23$).

\begin{figure}[h!] 
  \begin{minipage}[b]{0.5\linewidth}
    \centering
    \includegraphics[width=.98\linewidth]{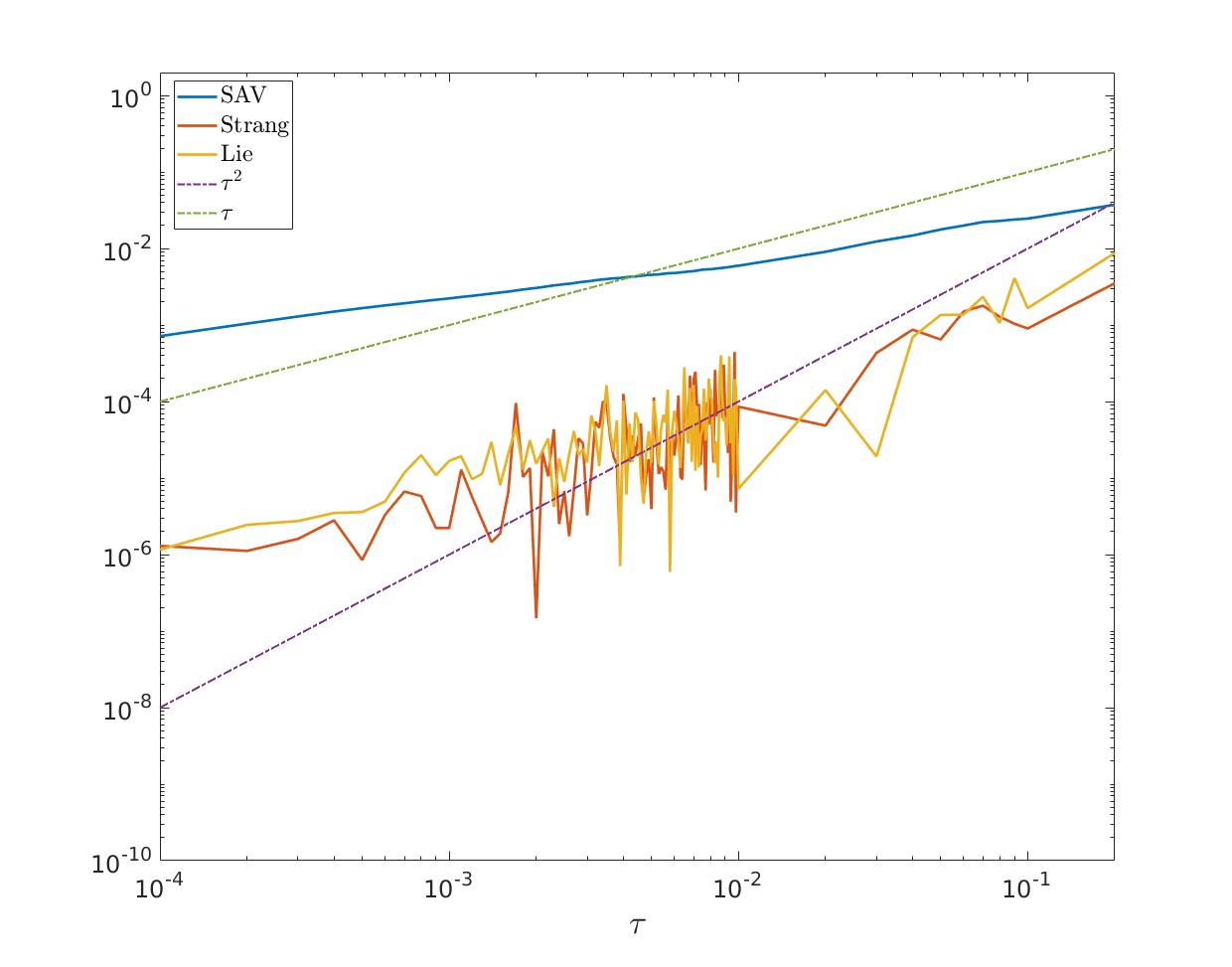} 
    \vspace{4ex}
  \end{minipage}
  \begin{minipage}[b]{0.5\linewidth}
    \centering
    \includegraphics[width=.98\linewidth]{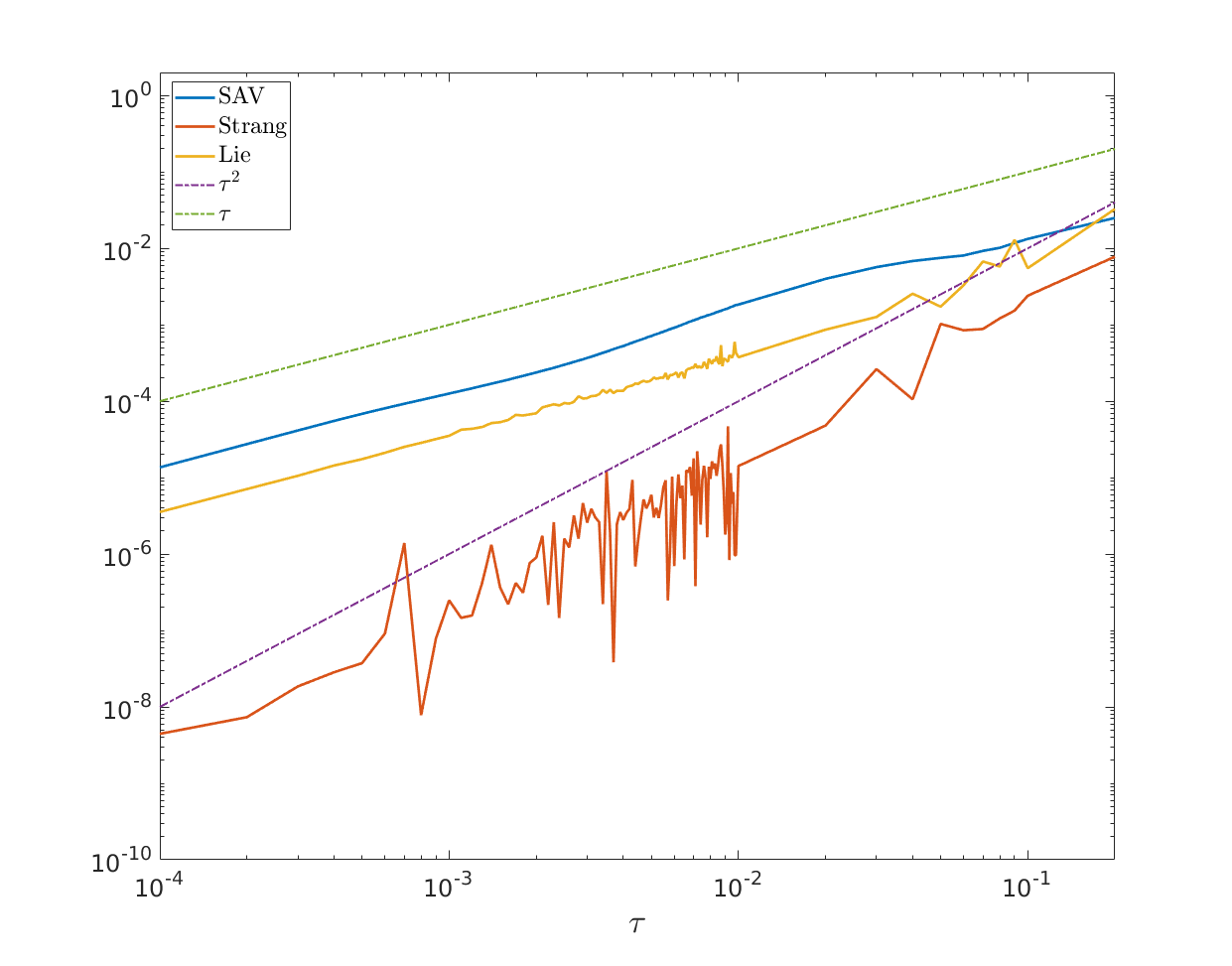} 
    \vspace{4ex}
  \end{minipage} 
  \begin{minipage}[b]{0.5\linewidth}
    \centering
    \includegraphics[width=.98\linewidth]{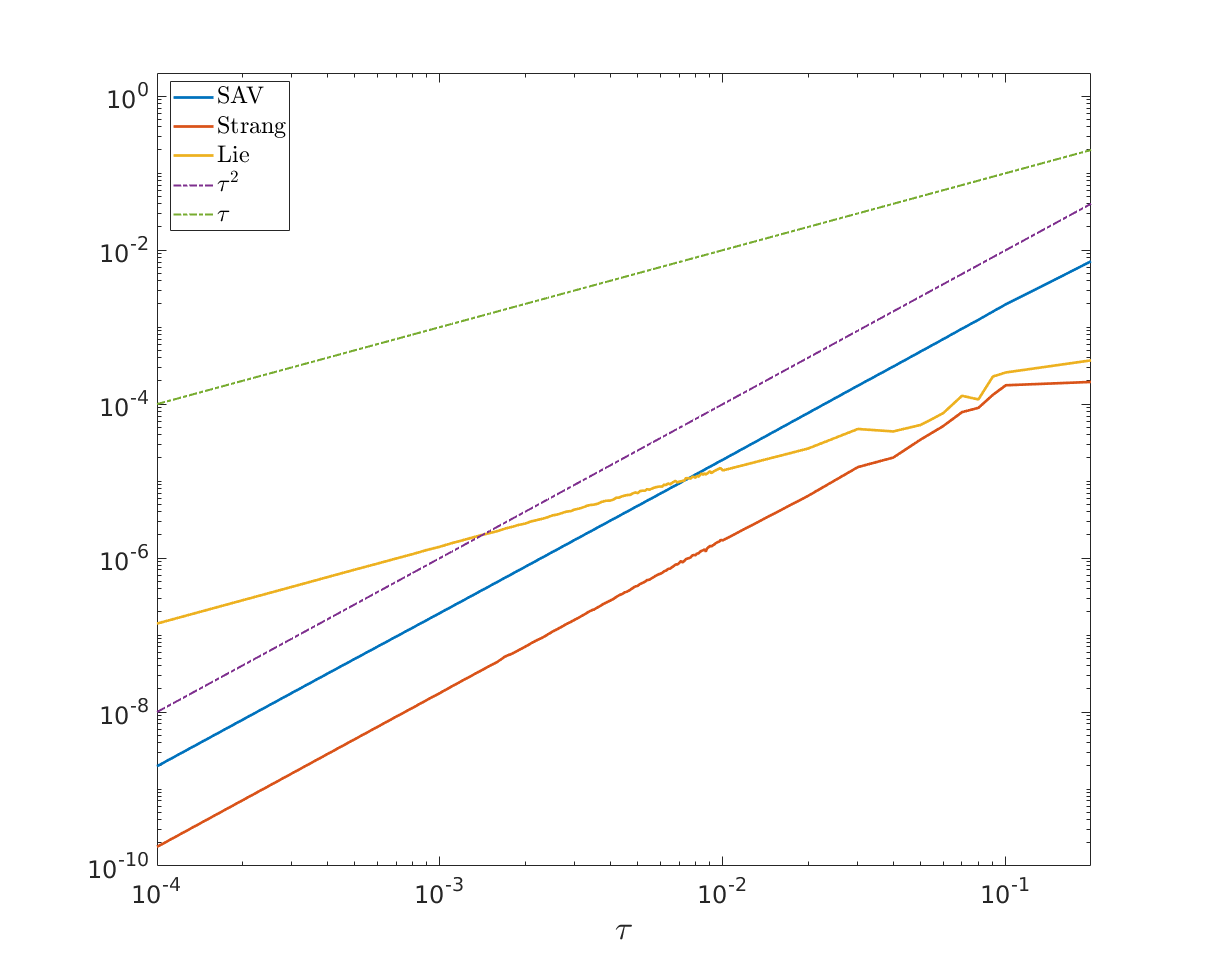} 
    \vspace{4ex}
  \end{minipage}
  \begin{minipage}[b]{0.5\linewidth}
    \centering
    \includegraphics[width=.98\linewidth]{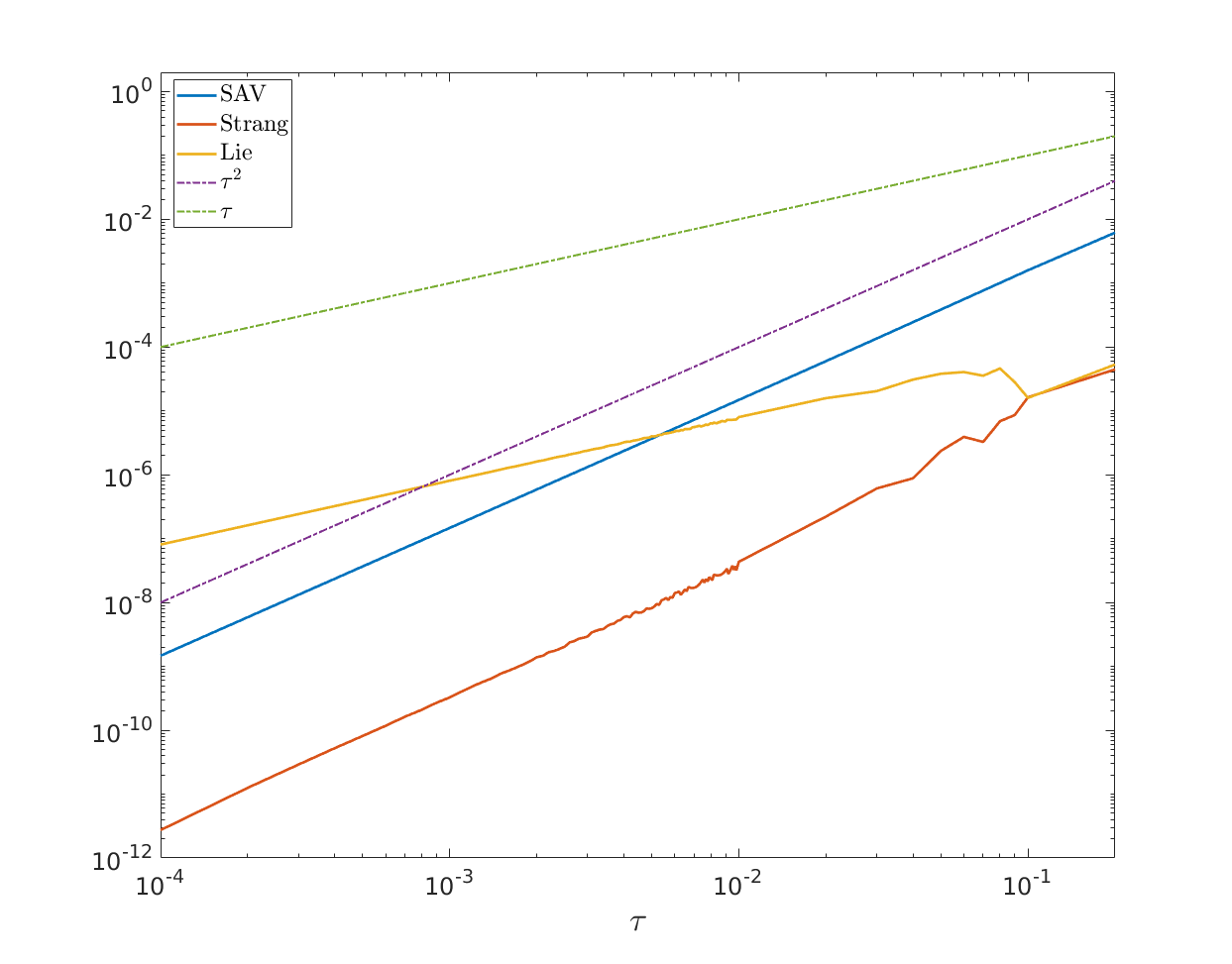} 
    \vspace{4ex}
  \end{minipage} 
  \caption{Error $e_H$ versus step size $\tau$ for the nonlinear Schrödinger equation starting from different initial conditions in $H^\alpha$ ($\alpha=\f23$ top-left, $\alpha=2$ top-right, $\alpha=3$ bottom left and $\alpha=5$ bottom-right). The dotted lines represent order $\tau$ (green) and $\tau^2$ (purple), respectively. }
   \label{fig:order-alpha} 
\end{figure}

\begin{figure}
    \centering
    \includegraphics[width=0.5\linewidth]{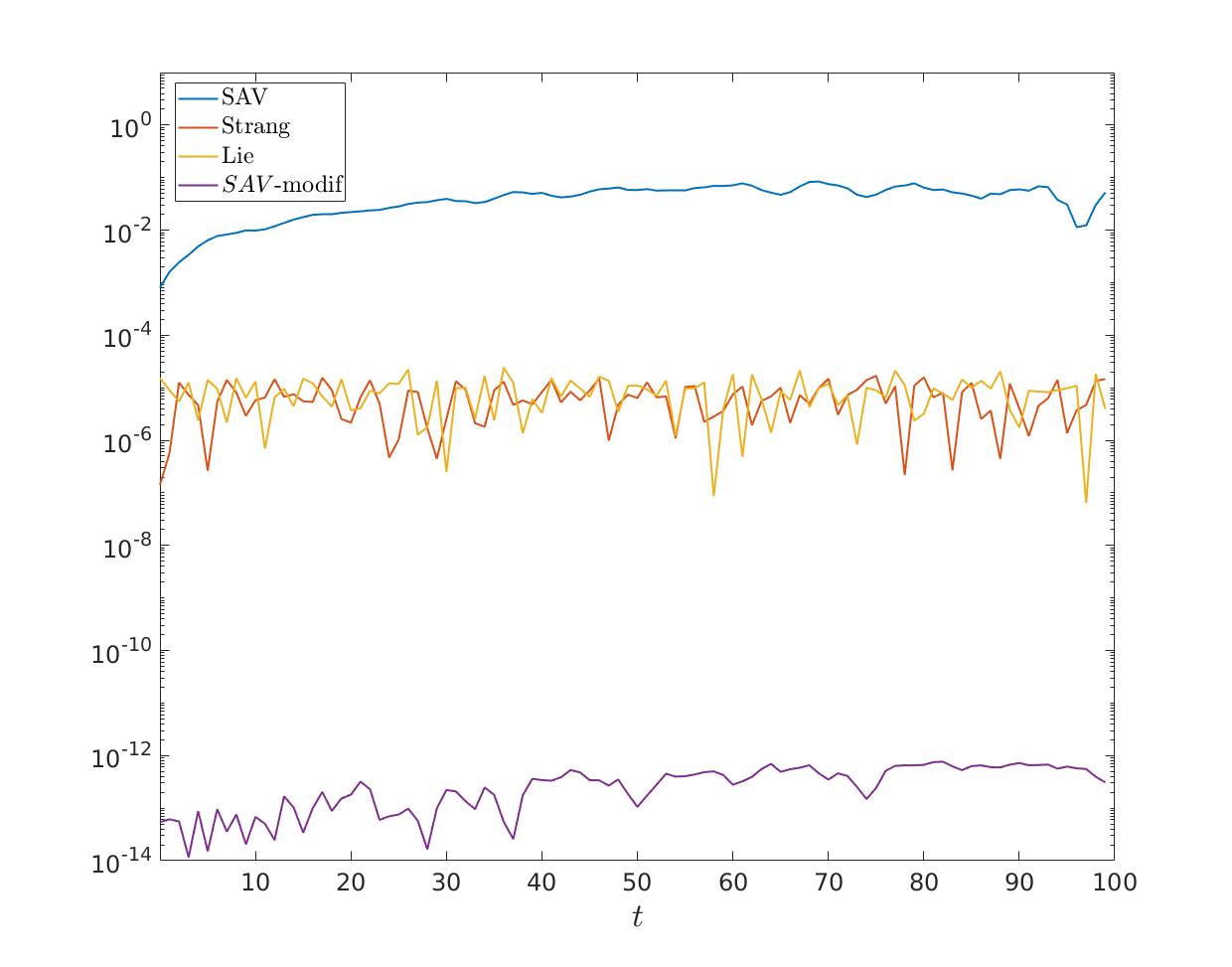}
    \caption{Error $e_H$ (blue, red, yellow) through time and $e_{\tilde H}$ (purple) through time for initial condition in $H^{\f23}$.}
    \label{fig:error-time-23}
\end{figure}

\subsection{Third test case: non-integer exponent}
In this example we consider the periodic nonlinear Schr\"odinger equation \eqref{eq:schro} with nonlinearities with non-integer exponents (\cite{crutcher_derivation_2011})
\[
    f(\abs{u}^2) = \beta \abs{u}^{4/\gamma},\quad \gamma >0
\]
where the Hamiltonian takes the form
\[
    H(u) = \int_\Omega \f12 \abs{\nabla u}^2 + \beta \frac{\gamma}{(4+2\gamma)}\abs{u}^{\f4\gamma+2} \dd x.
\]
We carry out simulations for various exponents $\gamma = 2,8/3,4,8$  up to time $T=10$ with smooth initial value
\[
    u(0,x) = \sin(x)\in C^\infty([-\pi,\pi]).
\]
The error $e_H$ for different exponents $\gamma$ is plotted in  Figure \ref{fig:order-exponent}. Our numerical findings suggest that  as $\gamma$ increases the splitting methods suffer from sever order reduction. This loss of convergence of splitting methods was also observed in \cite{schratz-low-2018}. The SAV method, on the other hand, retains its second order energy convergence for non-integer exponents.

\begin{figure}[h!] 
  \begin{minipage}[b]{0.5\linewidth}
    \centering
    \includegraphics[width=.98\linewidth]{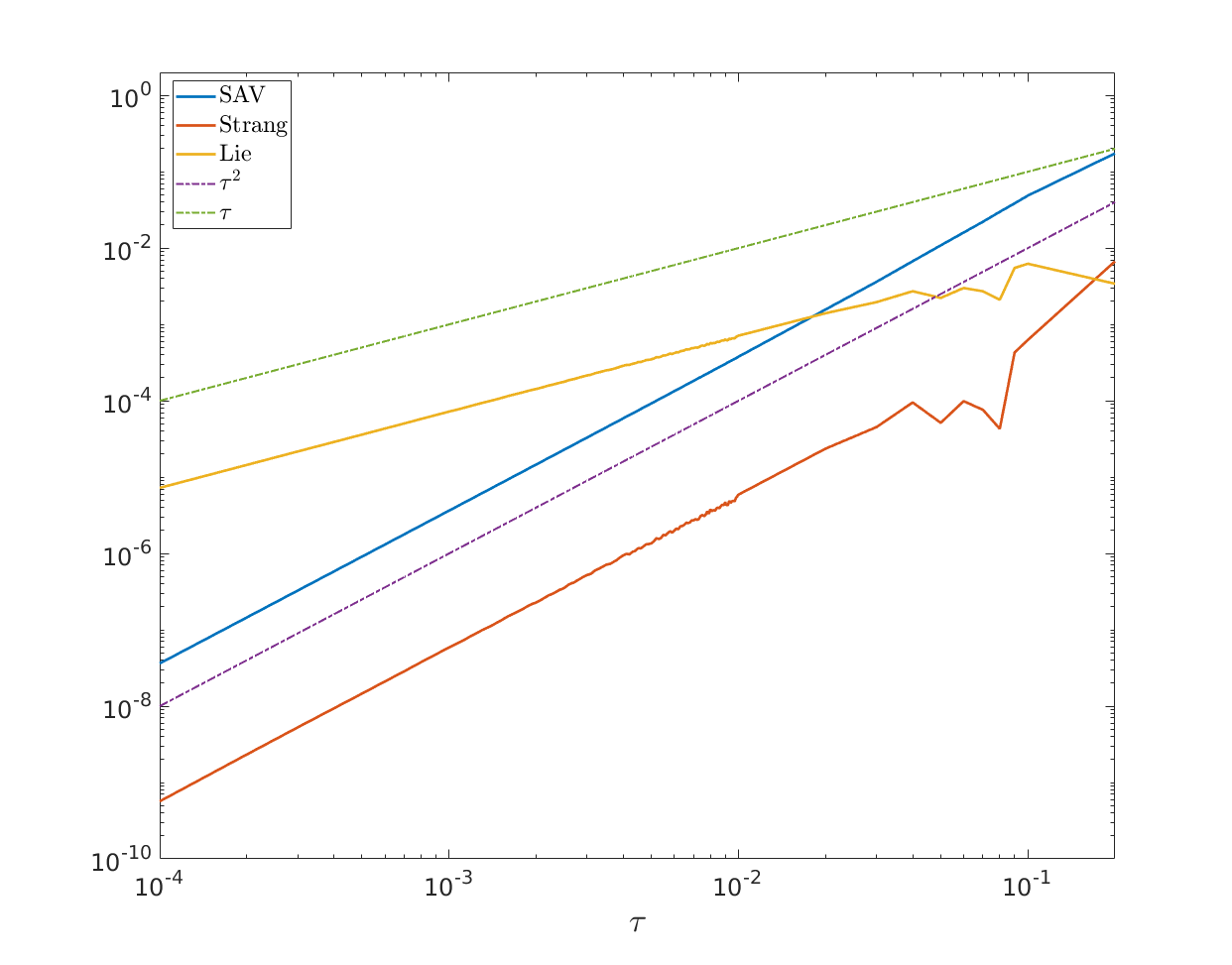} 
    \vspace{4ex}
  \end{minipage}
  \begin{minipage}[b]{0.5\linewidth}
    \centering
    \includegraphics[width=.98\linewidth]{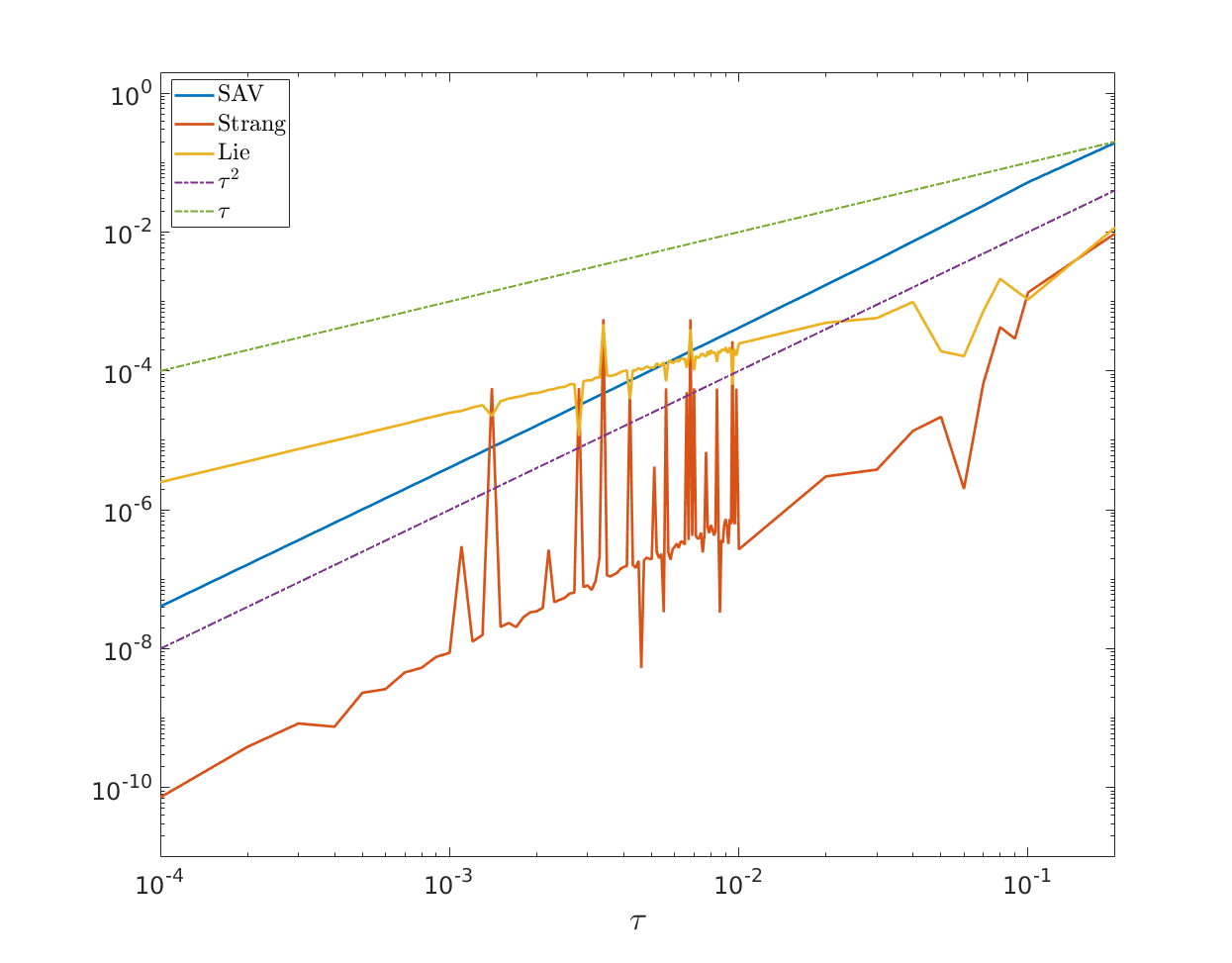} 
    \vspace{4ex}
  \end{minipage} 
  \begin{minipage}[b]{0.5\linewidth}
    \centering
    \includegraphics[width=.98\linewidth]{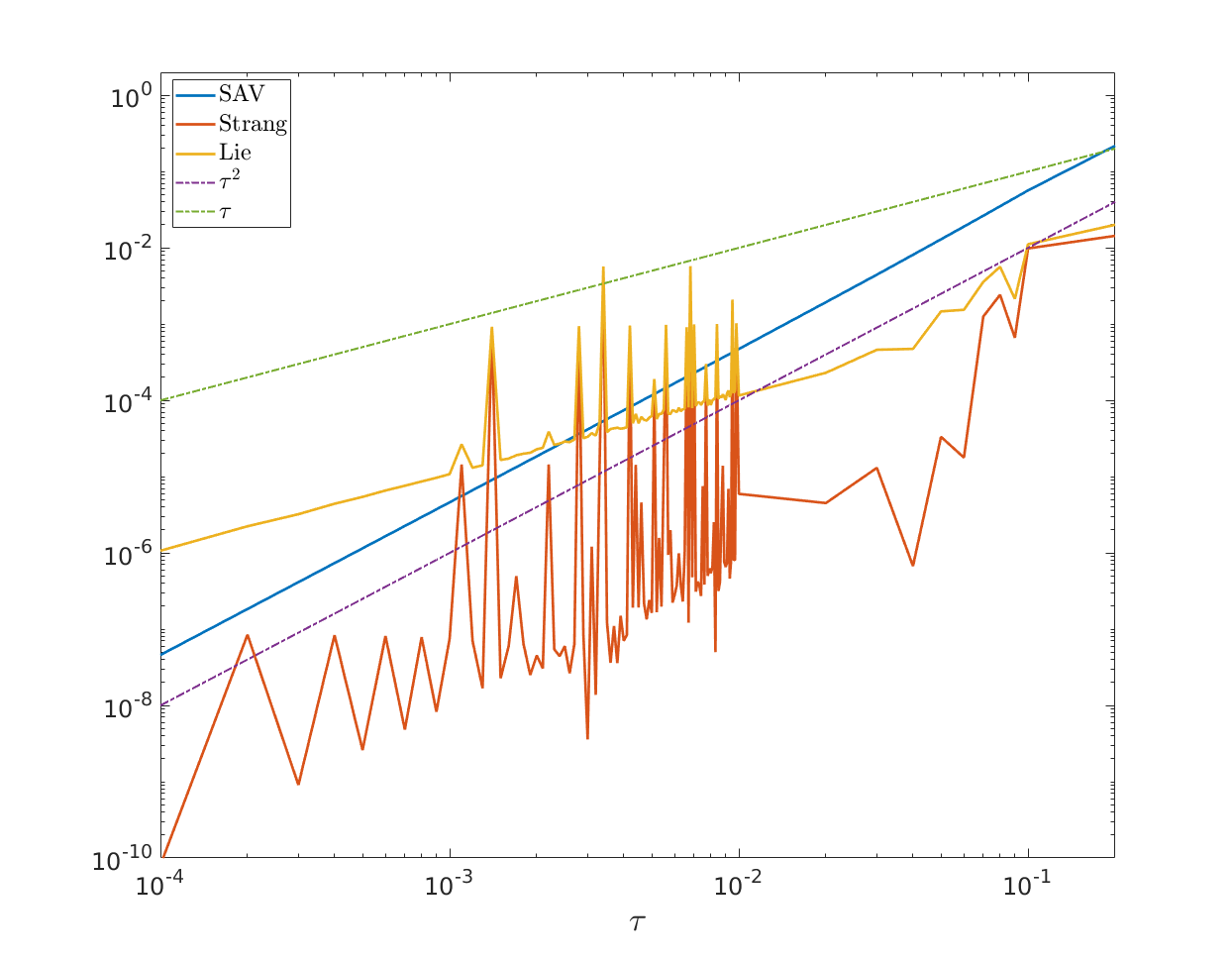} 
    \vspace{4ex}
  \end{minipage}
  \begin{minipage}[b]{0.5\linewidth}
    \centering
    \includegraphics[width=.98\linewidth]{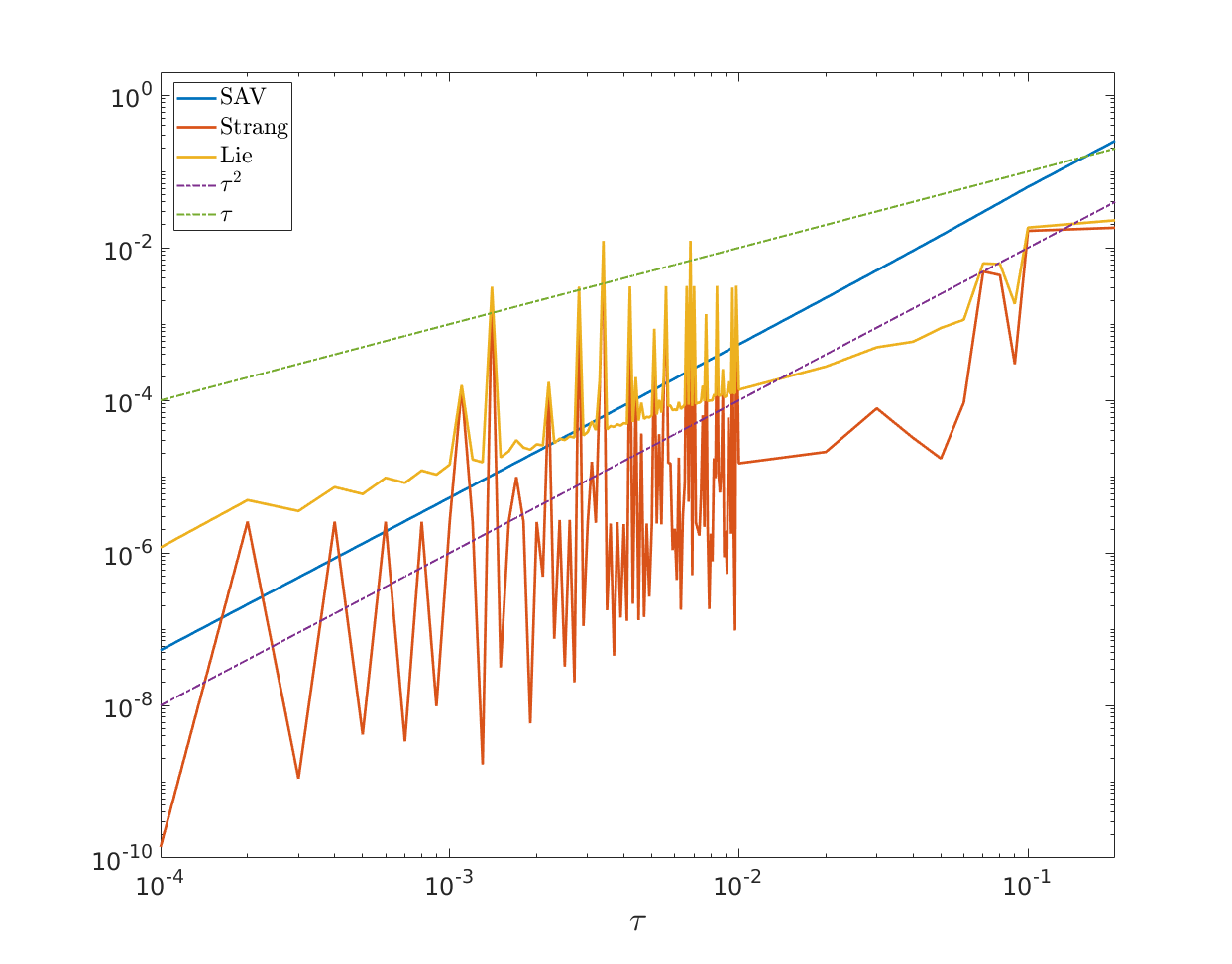} 
    \vspace{4ex}
  \end{minipage} 
  \caption{Error $e_H$ versus step size $\tau$ for the nonlinear Schrödinger equation with different non-integer exponents ($\gamma=2$ top-left, $\gamma=\f83$ top-right, $\gamma=4$ bottom left and $\gamma=8$ bottom-right). The dotted lines represent order $\tau$ (green) and $\tau^2$ (purple), respectively. }
   \label{fig:order-exponent} 
\end{figure}

\subsection{Computing ground states}
\AP{We use the SAV scheme to simulate ground states of Bose-Einstein condensates, however, unlike the work of Antoine \textit{et al.}~\cite{antoine-SAV-2020}, we propose here to use a different strategy. Indeed, in~\cite{antoine-SAV-2020}, the authors use the SAV scheme presented in Section~\ref{sec:schemes} and observe the capacity of the scheme to preserve the initial mass and Hamiltonian for various strengths of the nonlinearity. In the present work, we propose to use a different method and compare the numerical results with reference methods that are designed to simulate the stationary states of the NLS equation for large nonlinearities.
}

\AP{
To do so, we reformulate the problem into the solving of a gradient flow equation to compute these stationary states: this method is known has the \textit{gradient flow with discrete normalization} method~\cite{bao-ground, Bao-normalized}. }

\AP{
The SAV scheme is well adapted to this formulation since its original purpose was the simulation of the gradient flow equations. Details of the reformulation and the adaptation of the SAV scheme to the case can be found in Appendix~\ref{app:reformulation}. 
\begin{sloppypar}
We here present numerical results obtained choosing $d=1$, $V(x) = x^2/2$, $\beta = 400$, and ${u^0(x) = \f{\exp(-x^2/2)}{\pi^{1/4}} /\norm{u^0}_0}$. We validate our results with the Backward Euler PseudoSpectral (BEPS in short) scheme implemented in the GPELab code~\cite{ANTOINE20142969,ANTOINE201595}. 
\end{sloppypar}
We denote by $\cae(u)$ the energy associated to the renormalized system and $\tilde \cae(u)$ its modified SAV energy (see Appendix~\ref{app:reformulation} for details).
}

\AP{
Figure~\ref{fig:ground-state} (left) compares the stationary states obtained with the two schemes for $h=1/8$ in space. We clearly see that they both reach the same steady state. Figure~\ref{fig:ground-state} (right) depicts the evolution of the energy during the simulation. We observe that the SAV scheme preserves the monotonic decay of the energy. The steady state reached at the end of the simulation has an energy $\tilde \cae(\phi)\approx 22.90$. However, using the solution $\phi_g$ obtained for the SAV scheme and computing the "real" energy, we obtain $\cae(\phi)\approx 21.36$ which is the value obtained with the BEPS scheme. 
}

\begin{figure}[h!] 
  \begin{minipage}[b]{0.5\linewidth}
    \centering
    \includegraphics[width=.98\linewidth]{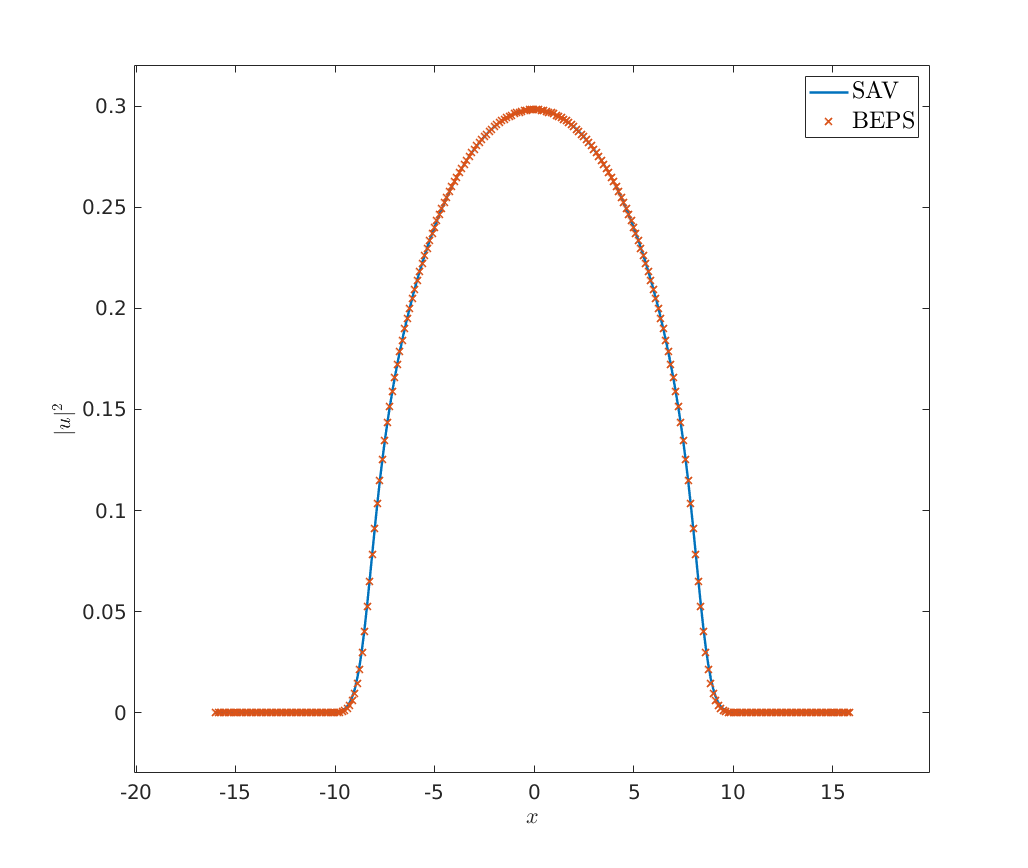} 
    \vspace{4ex}
  \end{minipage}
  \begin{minipage}[b]{0.5\linewidth}
    \centering
    \includegraphics[width=.98\linewidth]{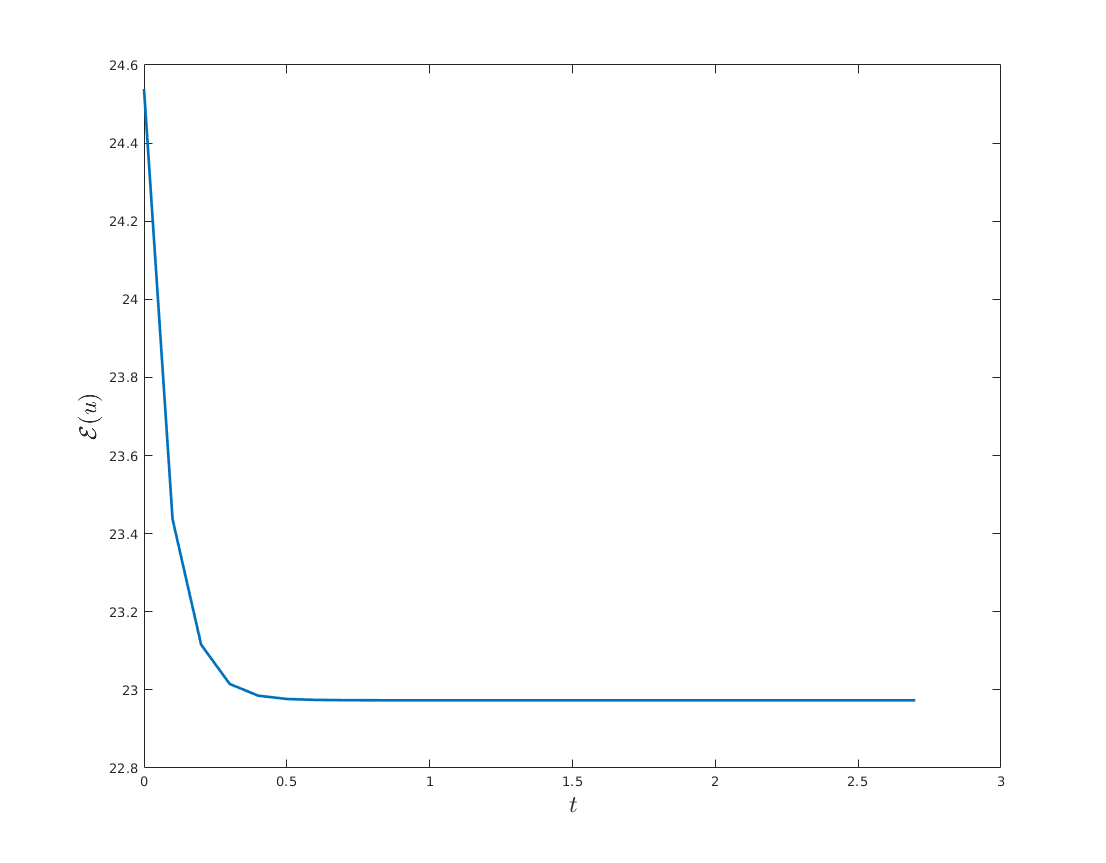} 
    \vspace{4ex}
  \end{minipage} 

  \caption{(Left) Comparison of stationary solutions of the NLS equation with a large cubic nonlinearity obtained with the SAV scheme (blue) and the BEPS scheme from GPELab (red). (Right) Evolution of the energy $\tilde \cae(u)$ for the SAV scheme during the simulation.}
   \label{fig:ground-state} 
\end{figure}

\begin{figure}
    \centering
    \includegraphics[width = 0.5\linewidth]{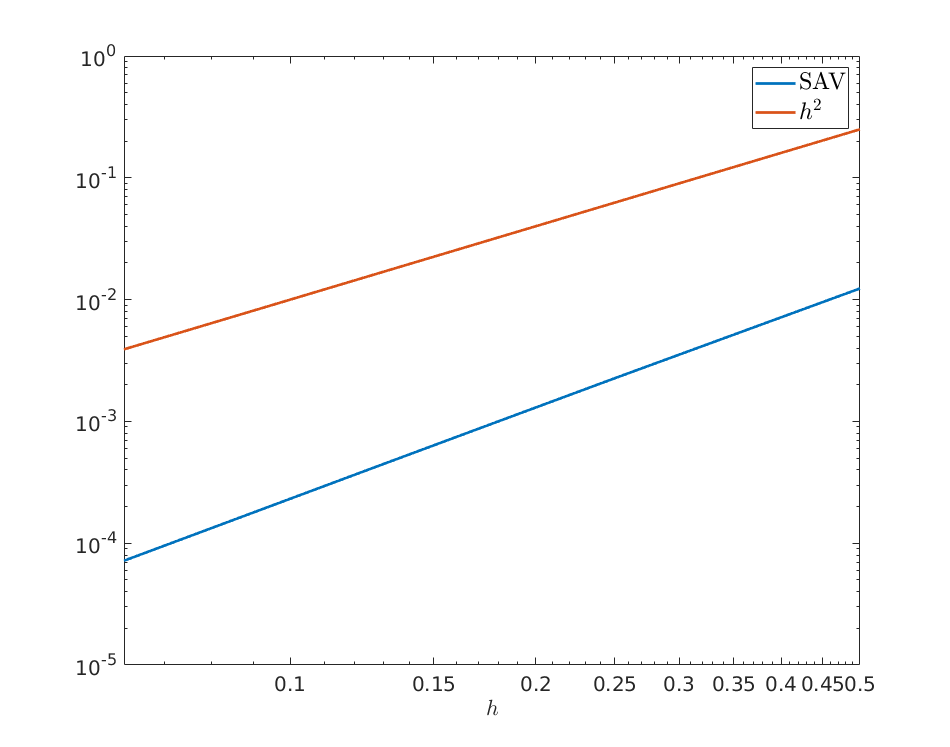}
    \caption{Error $e_u$ versus grid size $h$ for the simulation of ground states with a large cubic nonlinearity.}
    \label{fig:error-ground-state}
\end{figure}

\AP{
We numerically evaluate the order of convergence in space of the SAV scheme for the simulation of ground states. We choose our reference solution to be the result of a simulation with $h=1/32$. Then, we vary $h$ from $1/2$ to $1/16$. 
Figure~\ref{fig:error-ground-state} shows that the scheme remains second order convergent in space as predicted by our error analysis.
}

\appendix
\section{Gradient flow with discrete normalization for computing ground state} \label{app:reformulation}
\AP{A common method to compute stationary states of the NLS equation~\eqref{eq:schro} with a cubic nonlinearity is to write 
\[
    u(t,x) = \phi(x) \exp(-i\mu t),
\]
where $\mu$ is defined as the chemical potential of the condensate
\[
    \mu(\phi) = \int_\Omega\left(\f12 \abs{\nabla \phi} + \beta \abs{\phi}^4 + V(x)\abs{\phi}^2\right) \,\dd x.
\]
Therefore, using the previous reformulation in Equation~\eqref{eq:schro}, we obtain
\[
    \mu \phi(x) = -\f12 \Delta \phi(x) + \beta \abs{\phi(x)}^2\phi(x) + V(x) \phi(x).
\]
Denoting by $S=\{\phi | \norm{\phi}_{L^2(\Omega)} = 1\}$ the unit sphere, the ground state $\phi_g\in S$  of the Bose-Einstein condensate is then defined by the solution minimizing the energy functional  
\[
    \mathcal{E}(\phi) = \int_\Omega\left(\f12 \abs{\nabla \phi} + \f12\beta \abs{\phi}^4 + V(x)\abs{\phi}^2\right) \,\dd x < +\infty.
\]
For the proof of the existence of such state and other mathematical properties we refer the reader to \cite{Bao-Cai}.
}

\AP{
In the following, we adapt the Scalar Auxiliary Variable method to compute the stationary solutions of Equation~\eqref{eq:schro}. Therefore, endowing the equation with the normalization constraint, and using the projected gradient method~\cite{bao-ground}, the complete system reads
\[
\begin{cases}
    \p_t \phi = \f12 \Delta \phi - V(x)u - \beta\abs{u}^2 u(t,x),\\
    \norm{\phi}_{L^2(\Omega)}^2 = 1.
 \end{cases}
\]
Our SAV scheme can be easily adapted to this case, leading to the discrete system
\[
\begin{cases}
    \f{\phi^{+}-\phi^k}{\tau} =\f12 D^{(2)} \phi^{k+1/2} - r^{k+1/2}\tilde G^{k+1/2},\\
    r^{k+1}-r^k = \f12\scal{\tilde G^{k+1/2}}{\phi^{+}-\phi^k}\\
    \phi^{k+1} = \f{\phi^+}{\norm{\phi^+}_L^2(\Omega)},
\end{cases}
\]
with $\phi^{k+1/2} = \f{\phi^++\phi^k}{2}$, $\tilde G^{k+1/2}$ a second order approximation of $\f{\delta \cae_1[t^{k+1/2}]}{\delta \phi^{k+1/2}}$.
We precise that the associated modified SAV energy is 
\[
    \tilde \cae(\phi) = \int_\Omega \f12 \abs{\nabla \phi} \,\dd x + r(t)^2 < +\infty.
\]
}

\AP{
Both Algorithm 1 and Algorithm 2 from Section~\ref{sec:schemes} can be applied to compute the solution of the SAV system. Furthermore, using the same calculation as in Section~\ref{sec:apriori}, we can easily prove that the scheme dissipates the energy and preserves the normalization constraint.  }



%
%
%


\bibliographystyle{siam}  \bibliography{biblio}

\end{document}